\documentclass[noams]{compositio}
\usepackage{amssymb}
\usepackage{amsfonts}
\usepackage[all]{xy}
\usepackage{amsmath}
\usepackage{array}
\usepackage{pgf}
\usepackage{epsfig}  		
\usepackage{epic,eepic} 
\usepackage{graphicx}
\RequirePackage{fix-cm}
\usepackage{float}

\newtheorem{theorem}{Theorem}
\newtheorem{proposition}{Proposition}
\newtheorem{lemma}{Lemma}
\newtheorem{corollary}{Corollary}

\newtheorem{conjecture}{Conjecture}

\newtheorem{definition}{Definition}

\newtheorem{remark}{Remark}

\newcommand{\St}{\mathrm{St}}
\newcommand{\triv}{{\mathbf1}}

\begin{document}

\title{On the strict endoscopic part of modular Siegel threefolds
}

\author{Shervin Shahrokhi Tehrani}
\email{shervin@math.toronto.edu}

\address{Mathematics Department\\ University of Toronto\\ Bahen Centre
40 St. George St.\\ Toronto, Ontario\\
CANADA\\ 
M5S 2E4}



\begin{abstract}
In this paper we study the non-holomorphic strict endoscopic parts of  inner cohomology spaces of a modular Siegel threefold respect to local systems. First we show that there is a non-zero subspace of the strict endoscopic part such that it is constructed by global theta lift of automorphic froms of $(GL(2) \times GL(2))/\mathbb{G}_{m}$. Secondly, we present an explicit analytic calculation of levels of lifted forms into $GSp(4)$, based on the paramodular representations theory for $GSp(4, F)$. Finally, we prove the conjecture, by C. Faber and G. van der Geer, that gives a description of the strict endoscopic part for Betti cohomology and (real) Hodge structures in the category of mixed Hodge structures, in which the modular Siegel threefold has level structure one.

\end{abstract}

\maketitle

\setcounter{tocdepth}{1}

\section{Introduction}
\label{sec1}

Let $\mathfrak{H}_{1}$ denote the upper half plane and $\Gamma_{1}(N)$ is the the kernel of the natural homomorphism $GL(2, \mathbb{Z}) \to GL(2, \mathbb{Z}/ N\mathbb{Z})$. We know that $\mathcal{A}_{1, N}(\mathbb{C}):= \Gamma_{1}(N) \backslash \mathfrak{H}_{1} $  is the complex points of the moduli space of principally polarized abelian varieties of genus one with fixed level structure $N$. The geometry of $\mathcal{A}_{1, N}$ is entirely important in arithmetic geometry. One way of studying the geometric properties of $\mathcal{A}_{1, N}$ is considering the local systems (locally constant sheaves) on $\mathcal{A}_{1, N}(\mathbb{C})$ and studying the mixed Hodge structures on the cohomology spaces of $\mathcal{A}_{1, N}(\mathbb{C})$ respect to the local systems. This approach gives us information on  the Betti cohomology and (real) Hodge structures of  $\mathcal{A}_{1, N}$. Note that a local system can be parametrized by an integer $k$ and denoted by $\mathbb{V}(k)$ (see \cite{3}). The integer $k$ is the highest weight of a finite dimensional representation $(\rho, V_{\rho})$ of $GL(2)$. Then $\mathbb{V}(k)$ is $\Gamma_{1}(N) \backslash (\mathfrak{H}_{1} \times V_{\rho})$, where $\Gamma_{1}(N)$ acts on $\mathfrak{H}_{1} \times V_{\rho}$ as $\gamma.(\tau, v) = (\gamma\tau, \rho(\gamma)v)$ for any $\gamma \in \Gamma_{1}(N)$, $\tau \in \mathfrak{H}_{1}$, and $v \in V_{\rho}$. \\

Fix a local system $\mathbb{V}(k)$ on $\mathcal{A}_{1}$. For simplicity, we assume that the level structure is trivial, i.e., $N = 1$. Let $H^{1}(\mathcal{A}_{1}(\mathbb{C}), \mathbb{V}(k))$ and $H^{1}_{c}(\mathcal{A}_{1}(\mathbb{C}), \mathbb{V}(k))$ denote the cohomology and compactly supported cohomology spaces. Moreover, there is a natural map $f_{!}: H^{1}_{c}(\mathcal{A}_{1}(\mathbb{C}), \mathbb{V}(k)) \to H^{1}(\mathcal{A}_{1}(\mathbb{C}), \mathbb{V}(k))$. The inner cohomology is defined as    
\begin{equation*}
H^{1}_{!}(\mathcal{A}_{1}(\mathbb{C}), \mathbb{V}(k)) = im(f_{!}) = f_{!}(H^{1}_{c}(\mathcal{A}_{1}(\mathbb{C}), \mathbb{V}(k))).
\end{equation*}

\noindent It is well-known that the inner cohomology carries a pure Hodge structure of weight one that is describable in terms of holomorphic and anti-holomorphic terms as follows.

\begin{equation*}
H^{1}_{!}(\mathcal{A}_{1}(\mathbb{C}), \mathbb{V}(k)) = H^{1, 0}_{!}(\mathcal{A}_{1}(\mathbb{C}), \mathbb{V}(k)) \oplus H^{0, 1}_{!}(\mathcal{A}_{1}(\mathbb{C}), \mathbb{V}(k)).
\end{equation*}

\noindent This Hodge structure has a well-known description in terms of modular forms of $SL(2, \mathbb{Z})$. Explicitly, if $S_{k}$ denotes the space of cusp forms of $SL(2, \mathbb{Z})$ of weight $k$, then the Eichler-Shimura theorem says that for an even $k \in \mathbb{Z}$, we have an isomorphism of inner cohomology 

\begin{equation}
\label{ES}
H^{1}_{!}(\mathcal{A}_{1}(\mathbb{C}), \mathbb{V}(k)) \cong S_{k + 2} \oplus \overline{S}_{k + 2},
\end{equation}
where $\overline{S}_{k + 2}$ denotes the anti-holomorphic modular forms of weight $k + 2$.\\

\noindent The Eichler-Shimura theorem gives us an explicit connection between geometric and arithmetic properties of $\mathcal{A}_{1}$. Now we can consider the locally symmetric space $\mathcal{A}_{g, N}:= \Gamma_{g}(N) \backslash \mathfrak{H}_{g}$ ($g \geq 2$). Its complex points is the moduli space of principally polarized abelian varieties of genus $g$. It is desirable to have a generalization of Eichler-Shimura theorem for an arbitrary genus $g$. The local systems and inner cohomology are generalized similarly as in the case $g = 1$ (see Section \ref{sec2}). Now if $\mathbb{V}(\lambda)$ is a local system on $\mathcal{A}_{g, N}$, then the inner cohomology $H_{!}^{d}(\mathcal{A}_{g, N}, \mathbb{V}(\lambda))$ carries a pure Hodge structure of weight $d$, where $d = \dim(\mathcal{A}_{g, N}(\mathbb{C}))$\footnote{Technically, if $\lambda$ is non-singular, then $H_{!}^{d}(\mathcal{A}_{g, N}, \mathbb{V}(\lambda))$ is the only non-zero cohomolgical space (See \cite{1}). Therefore, we need only to consider the inner cohomology of degree $d$. }. Moreover, this Hodge structure described explicitly in terms of spaces of cusp forms of $GSp(2g, \mathbb{Z})^{+}$ (see \cite{2}, p. 237). For example, the Hodge structure in the case $g = 2$ can be described in greater details as follows.\\

 Let $G$ denote the algebraic similitude symplectic group $GSp(4)$ over $\mathbb{Q}$. If $(\rho, V_{\rho})$ is an irreducible finite dimensional representation of $G$ with highest weight $\lambda$, then $\Gamma_{2}(N)$\footnote{$\Gamma_{2}(N)$ is the the kernel of the natural homomorphism $Sp(4, \mathbb{Z}) \to Sp(4, \mathbb{Z}/ N\mathbb{Z})$. Also, $\mathfrak{H}_{2}$ denotes the Siegel upper half space of genus 2.} acts on $\mathfrak{H}_{2} \times V_{\rho}$ as $\gamma.(\tau, v) = (\gamma\tau, \rho(\gamma)v)$ for any $\gamma \in \Gamma_{2}(N)$, $\tau \in \mathfrak{H}_{2}$, and $v \in V_{\rho}$. Therefore, the quotient space $\Gamma_{2}(N)\backslash (\mathfrak{H}_{2} \times V_{\rho})$ is a locally constant sheaf on $\mathcal{A}_{2, N}(\mathbb{C})$ denoted by $\mathbb{V}(\lambda)$. $\mathbb{V}(\lambda)$ is called the local system on $\mathcal{A}_{2, N}$ of the parameter $\lambda$.\\
  
\noindent Consider the local system $\mathbb{V}(\lambda)$ on $X_{N}:= \mathcal{A}_{2, N}$ such that $\lambda = (l, m)$ (see Section \ref{sec2}) be sufficiently regular (i.e., $l - m  \geq 2, m \geq 2 $). Now consider $H_{!}^{3}(X_{N}, \mathbb{V}(\lambda))$ (i.e., $d = 3$). According to the theory of automorphic vector bundles (\cite{12} and \cite{8}), we have an isomorphism

\begin{equation}
\label{1000}
H^{3}_{!}(X_{N}, \mathbb{V}(\lambda)) \cong \bigoplus_{\pi = \pi_{\infty} \otimes \pi_{fin}} m(\pi) (H^{3}(\mathfrak{g}, K, \pi_{\infty}) \otimes (\pi_{fin})^{K(N)}),
\end{equation}
where the sum runs over all irreducible cuspidal automorphic representations such that $\pi_{\infty}$ belongs to the set of discrete $(\mathfrak{g}, K)$-modules with infinitesimal Harish-Chandra character $\chi_{\lambda + \rho}$. $(\pi_{fin})^{K(N)}$ denotes the invariant elements of $\pi_{fin}$ under the action of open compact subgroup $K(N)$ defined in  (\ref{level}). $H^{3}(\mathfrak{g}, K, \pi_{\infty})$ is the $(\mathfrak{g}, K)$-cohomology defined for example in  \cite{40} and $m(\pi)$ denotes the multiplicity of $\pi$ in discrete spectrum. According to the fact that  $H^{3}(\mathfrak{g}, K, \pi_{\infty})$ is a  1-dimensional space (see \cite{41}), with respect to a fixed $(\mathfrak{g}, K)$-class basis for $H^{3}(\mathfrak{g}, K, \pi_{\infty})$, we can simplify (\ref{1000}) to

\begin{equation}
\label{200}
H^{3}_{!}(X_{N}, \mathbb{V}(\lambda)) \cong \bigoplus_{\pi = \pi_{\infty} \otimes \pi_{fin}} m(\pi) (\pi_{fin})^{K(N)}.
\end{equation}
 There are four different types of discrete series representations with infinitesimal Harish-Chandra character $\chi_{\lambda + \rho}$ for $GSp(4, \mathbb{R})^{+}$. These four discrete series are quite different as $(\mathfrak{g}, K)$-modules, where $\mathfrak{g}$ denotes the Lie algebra of $GSp(4, \mathbb{R})$ and $K = Z(\mathbb{R})U(2)$.\\  

Based on the pure Hodge structure on $H_{!}^{3}(X_{N}, \mathbb{V}(\lambda))$, there is a decreasing Hodge filtration 
\begin{equation}
\label{201}
 (0) \subset F^{3} \subset F^{2} \subset F^{1} \subset F^{0} = H_{!}^{3}(X_{N}, \mathbb{V}(\lambda)),  
\end{equation}
where $F^{r} = \bigoplus_{p \geq r} H_{!}^{p, q}(X, \mathbb{V}(\lambda)) $. In \cite{2}, Faltings-Chai present an explicit description of each term of the above filtration as cohomological spaces of $X_{N}$ respect to holomorphic vector bundles on $X_{N}$. This provides that the $F^{i}$ has a description in terms of vector valued Siegel modular forms of genus 2. For example, the holomorphic part of the Hodge filtration, $F^{3}$, is isomorphic with $S_{l - m, m + 3}$. Here $S_{j, k}$ denotes the space of cuspidal Siegel modular forms corresponded to the representation $Sym^{j} \otimes \det^{k}$ of $GL(2, \mathbb{C})$ (i.e., Siegel modular form of genus 2).\\

\noindent This is a generalization of Eichler-Shimura theorem to the case $g = 2$. In contrast to $g = 1$, in which we have only the holomorphic and anti-holomorphic terms, here there exist non-holomorphic terms $H_{!}^{1, 2}(X_{N}, \mathbb{V}(\lambda))$ and $H_{!}^{2, 1}(X_{N}, \mathbb{V}(\lambda))$. However, the $F^{i}$ term was described as a space of Siegel modular forms of $GSp(4)$. Since the theory of Siegel modular forms of genus $g \geq 2$ is not well understood as much as the elliptic modular forms (i.e., genus one modular forms), it is desirable to construct the cohomology classes in $H_{!}^{3}(X_{N}, \mathbb{V}(\lambda))$ in terms of elliptic modular forms by applying global theta lifting.\\ 

The classical approach to study elliptic modular forms was based on the analytic methods applied on the Fourier expansion of modular forms. This approach is completely effective when we study the modular forms on the upper half plane $\mathfrak{H}_{1}$ (See \cite{3}, Chapter one). This leads to fully understand the properties of the elliptic modular forms. However, if we want to study the Siegel modular forms on Siegel upper half space $\mathfrak{H}_{g}$ ($g \geq 2)$, the analytic approach is not effective because of difficulty on calculation of their Fourier expansions. Since the theory of Siegel modular forms on $\mathfrak{H}_{1}$ is well-known, an effective approach to study elements in $F^{i}$ is trying to construct them in terms of lifting of cuspidal automorphic forms of $GSp(2)$.\\

The first term in (\ref{201}), $F^{3} = H^{3, 0}_{!}(X_{N}, \mathbb{V}(\lambda))$ corresponds to the space of holomorphic Siegel modular forms. This part have been studied by T. Oda and J. Schwermer in \cite{25} and  S. B\"{o}cherer and R. Schulze-Pillot in \cite{28}. They construct non-zero cohomological classes via lifting theory in $F^{3}$ from the rank one algebraic group $O(4)$. Note that by conjugation property of the Hodge structure, their result for the holomorphic part can be considered for the anti-holomorphic part $H^{0, 3}_{!}(X, \mathbb{V}(\lambda))$ too. However, in this paper we are going to focus on the middle terms in (\ref{201}). Explicitly, we want to study the classes in $\mathbb{H}_{\text{non-holomorphic}} =  H^{1, 2}_{!}(X_{N}, \mathbb{V}(\lambda)) \oplus H^{2, 1}_{!}(X_{N}, \mathbb{V}(\lambda))$. This part is the direct sum over $(\pi_{fin})^{K(N)}$ of $\pi = \pi_{\infty} \otimes \pi_{fin}$, where $\pi_{\infty}$ is one of the two non-holomorphic discrete series representations of $GSp(4, \mathbb{R})^{+}$ with infinitesimal Harish-Chandra character $\chi_{\lambda + \rho}$ (see Section \ref{sec2}). There are some $\pi_{fin}$'s which contribute in all four pieces of Hodge filtration. It means such a $\pi_{fin}$ appears as a finite part of cuspidal automorphic representations  of $GSp(4, \mathbb{A})$ with all different four types of discrete series representations with infinitesimal Harish-Chandra character $\chi_{\lambda + \rho}$ in their archimedean part. In other words, $\pi_{fin}$ contributes in (\ref{200}) with positive $m(\pi)$ where $\pi_{\infty}$ runs over all four distinct discrete $(\mathfrak{g}, K)$-modules. They are called stable classes. But there are other $\pi_{fin}$'s that contribute either in holomorphic part or nonholomorphic part. These classes are called endoscopic cohomological classes. As mentioned before, the holomorphic endoscopic classes are well understood from the automorphic point of view, whereas a structural description of the non-holomorphic endoscopic classes is still open. In this paper we are going to present such a structural description for the non-holomorphic endoscopic classes in terms of automorphic forms of $(GL(2) \times GL(2))/ \mathbb{G}_{m}$.\\

Explicitly, we define 
\begin{equation}
\label{202}
\mathcal{C} = \lbrace \pi = \pi_{\infty} \otimes \pi_{fin} \mid \pi \quad \text{is irreducible and} \quad H^{3, 0}_{!}(X_{N}, \mathbb{V}(\lambda))(\pi_{fin})^{K(N)} = 0 \rbrace.
\end{equation}
Now let
\begin{equation}
\label{I7}
H_{\text{End}^{\text{s}}}^{3}(X, \mathbb{V}(\lambda)) = \bigoplus_{\pi \in \mathcal{C}  }  m(\pi)(\pi_{fin})^{K(N)}, 
\end{equation} 
which is called the strict endoscopic part. This part of inner cohomology is not detectable by holomorphic part of Hodge structure.  According to numerical calculation, authors conjecture in \cite{24} the following structural description.

\begin{conjecture}
\label{conj1}
The strict endoscopic part contributes as a finite part of $\theta$-lifted of cuspidal automorphic representations of $(GL(2) \times GL(2))/ \mathbb{G}_{m}$ over $\mathbb{Q}$.
\end{conjecture}

\noindent Moreover, if we consider the local system $\mathbb{V}(\lambda)$ on $\mathcal{A}_{2} := \mathcal{A}_{2, 1}$, the Conjecture \ref{conj1} can be stated in greater details. In \cite{22}, C. Faber and G. van der Geer, based on numerical calculations, conjecture an explicit description of the strict endoscopic part of a local system $\mathbb{V}(l, m)$ on $\mathcal{A}_{2}$, where $l > m > 0$. \\

\noindent Let $S_{k} = S_{k}(\Gamma_{1}(1))$ denote the space of cusp forms on $\mathfrak{H}_{1}$. Following A. Scholl (\cite{23}), we can associate a motive to this space denoted by $S[k]$. Moreover, $s_{k} = \dim(S_{K})$. Also, $\mathbb{L}$ is the 1-dimensional Tate Hodge structure of weight 2. It is the Tate motive corresponds to the second cohomology of $\mathbb{P}^{1}$.\\

\noindent Let 
\begin{equation*}
e(A_{2}, \mathbb{V}(\lambda)) = \sum_{i}(-1)^{i}[H^{i}(\mathcal{A}_{2}, \mathbb{V}(\lambda))],
\end{equation*}
denote the motivic Euler characteristic of  $H^{i}(\mathcal{A}_{2}, \mathbb{V}(\lambda))$, where this expansion is taken in the Grothendieck group in the category of mixed Hodge structure. Similarly, Let 
\begin{equation*}
e_{c}(\mathcal{A}_{2}, \mathbb{V}(\lambda)) = \sum_{i}(-1)^{i}[H^{i}_{c}(\mathcal{A}_{2}, \mathbb{V}(\lambda))],
\end{equation*}
denote  the motivic Euler characteristic of  $H^{i}_{c}(\mathcal{A}_{2}, \mathbb{V}(\lambda))$. The kernel of $f_{!}$ is called the Eisenstein cohomology. The corresponding motivic Euler characteristic denotes by $e_{\text{Eis}}(\mathcal{A}_{2}, \mathbb{V}(\lambda))$. Similarly, $e_{!}(\mathcal{A}_{2}, \mathbb{V}(\lambda))$ denotes the corresponding motivic Euler characteristic of the inner cohomology. Based on the pure Hodge structure on $H_{!}^{3}(\mathcal{A}_{2}, \mathbb{V}(\lambda))$, there is a decreasing Hodge filtration 
\begin{equation*}
 (0) \subset F^{3} \subset F^{2} \subset F^{1} \subset F^{0} = H_{!}^{3}(\mathcal{A}_{2}, \mathbb{V}(\lambda)),  
\end{equation*}
where $F^{r} = \bigoplus_{p \geq r} H_{!}^{p, q}(\mathcal{A}_{2}, \mathbb{V}(\lambda)) $. Under the isomorphism (\ref{5}), we have 

\begin{equation}
\label{I2}
H^{3}_{!}(\mathcal{A}_{2}, \mathbb{V}(\lambda)) \cong \bigoplus_{\pi = \pi_{\infty} \otimes \pi_{fin}} m(\pi)(\pi_{fin})^{K_{0}},
\end{equation}
where the sum runs over all irreducible cuspidal automorphic representations such that $\pi_{\infty}$ belongs to the discrete series with infinitesimal Harish-Chandra character $\chi_{\lambda + \rho}$ and $\pi_{fin}$ has the level $K_{0}$, where $K_{0}$ is the restricted product of $K_{\upsilon}$'s such that $K_{\upsilon} = GSp(4, \mathcal{O}_{\upsilon})$ (i.e., paramodular group $K(\mathcal{P}_{\upsilon}^{0})$). Now 
\begin{equation*}
\mathcal{C} = \lbrace \pi = \pi_{\infty} \otimes \pi_{fin} \in \mathcal{A}_{0}(G) \mid \pi \quad \text{is irreducible and} \quad H^{3}_{!}(\mathcal{A}_{2}, \mathbb{V}(\lambda))(\pi_{fin}) = 0 \rbrace.
\end{equation*}
Let
\begin{equation}
\label{611}
H_{\text{End}^{\text{s}}}^{3}(\mathcal{A}_{2}, \mathbb{V}(\lambda)) = \bigoplus_{\pi \in \mathcal{C}  }  m(\pi)(\pi_{fin})^{K_{0}}, 
\end{equation} 
which is the strict endoscopic part and its corresponding motivic Euler characteristic denoted by $e_{\text{endo}}(\mathcal{A}_{2}, \mathbb{V}(\lambda))$. Now we have

\begin{conjecture}[C. Faber and G. van der Geer]
\label{conj2}
With the above notation, we have
\begin{equation}
\label{Main}
e_{\text{endo}}(A_{2}, \mathbb{V}(\lambda)) = -s_{l + m + 4}S[l - m + 2]\mathbb{L}^{m + 1},
\end{equation}
where $\lambda = (l, m)$ and $\lambda$ is sufficiently regular.
\end{conjecture}

This paper provides a systematic approach to attack Conjecture \ref{conj1} and fully prove the equation (\ref{Main}) in Conjecture \ref{conj2} for the Betti cohomology and (real) Hodge structures. This approach contains an application of a key construction by M. Harris and S. S. Kudla in \cite{5} of cohomology classes of Siegel threefolds via theta lifting from $GO(2, 2)$. However, the non-vanishing of this construction is an important issue which we overcome on it by using a result by S. Takeda in \cite{10}. Moreover, because of using Harris-Kudla's result, we lose our control on the level structure of our cohomology classes. We solve this problem by an explicit investigation on the local theory for each place. Explicitly, we regain the level structure control by finding each local component of our structure. To do this, we use a result by W. Gan and S. Takeda in \cite{17} and figure out the level of lifted forms to $GSp(4)$ by a result of B. Roberts and R. Schmidt in \cite{16}.\\ 

\noindent Technically, Corollary \ref{c4} gives a structural description of a non-zero subspace of the strict endoscopic part constructed via $\theta$-lifting with determined dimension. However because there is no determined formula for dimension of inner cohomology, our approach can not prove the first conjecture completely. Since we know that the Euler characteristic of inner cohomology (\cite{22}) for $\mathcal{A}_{2, 1}$ and have Tsushima's dimension formula (\cite{26} and \cite{27}) we prove the Conjecture \ref{conj2} in the Section \ref{sec9}. Note that I don't claim to prove Conjecture \ref{conj2} in the category of motives, but we show that for Betti cohomology and (real) Hodge structures in the category of mixed Hodge structures the equation (\ref{Main}) satisfies. This could be an evidence to strengthen the truthfulness of Conjecture \ref{conj2} in the category of motives. Moreover Section \ref{sec7} has an explicit analytic calculation of level of global theta lifted automorphic representations to $GSp(4)$ based on theory of local newforms of $GSp(4)$ developed by B. Roberts and R. Schmidt in \cite{16}. This explicit analytic results could provide so many important information on the global theta lift to $GSp(4)$. For example, in Section \ref{sec10}, we revisit our results in terms of Hecke eigenvalues.\\

This paper is organized as follows. Section \ref{sec2} recalls the definition of local systems on $X_{N}$. Moreover, an explicit description of the Hodge filtration terms is presented in Table \ref{table1}. Sections \ref{sec3} and \ref{sec4} review the theory of automorphic vector bundles as in \cite{5}. In these sections we present the description of cohomological classes in terms of finite part of cuspidal automorphic representations. Section \ref{sec5} restates the global theta correspondence for the similitude dual reductive pair $(GO(2, 2), GSp(4))$ in our desired way. This section also contains the non-vanishing criteria of the global theta lift. In Section \ref{sec6} we construct the strict endoscopic classes and state the first main result in Theorem \ref{th5}. Section \ref{sec7} provides an explicit analytic analysis of the level of lifted admissible representations into $GSp(4, F)$, where $F$ is a non-archimedean local field. To clarify the complicated calculation in Section \ref{sec7}, Section \ref{sec8} involves several practical examples. In Section \ref{sec9} we give the proof of Conjecture \ref{conj2} according to Corollary \ref{c5}. Finally, Section \ref{sec10} contains the interpretation of our results in terms of Hecke Operators. 

\subsection*{Notation} 
This part contains some general notation needed in this paper.
Let $\Gamma$ be an arithmetic subgroup of the symplectic group  $Sp(4, \mathbb{Z})$. Let $\mathfrak{H}_{2}$ denote the Siegel upper half space $\lbrace Z = Z^{t} \in Mat_{2, 2}(\mathbb{C}) \mid Im(Z) > 0\rbrace$ of genus $2$. Matrices $\begin{pmatrix}
A & B\\
C & D
\end{pmatrix} \in GSp(4, \mathbb{R})$ act on $\mathfrak{H}_{2}$ by linear fractional action $M.Z = (AZ + B)(CZ + D)^{-1}$. Let $X_{\Gamma}$ denote $\Gamma \backslash \mathfrak{H}_{2}$ and $\overline{X}_{\Gamma}$ denote an appropriate toroidal compactification of $X_{\Gamma}$. Moreover, $j: X_{\Gamma} \to \overline{X}_{\Gamma}$ denotes the natural inclusion map. If $\mathcal{W}$ denotes a holomorphic vector bundle on $X_{\Gamma}$, then it has a natural extension on $\overline{X}_{\Gamma}$ denoted by $\overline{\mathcal{W}}$ (see \cite{2}, Thm 4.2., p. 255).\\

If $G$ is an algebraic group over $\mathbb{Q}$, then $G(\mathbb{R})$ denotes its real points and $G(\mathbb{R})^{+}$ means the connected component of $G(\mathbb{R})$. Let $\mathcal{A}(G)$ and $\mathcal{A}_{0}(G)$ denote the spaces of automorphic functions and cuspidal automorphic functions on $G(\mathbb{A})$, where $\mathbb{A}$ is the adelic ring of $\mathbb{Q}$. We denote the finite part of $\mathbb{A}$ by $\mathbb{A}_{f}$.\\

Let $(\pi, V_{\pi})$ be an automorphic representations of $G(\mathbb{A})$. We always assume that $V_{\pi} \subset \mathcal{A}(G)$ and $\pi = \pi_{\infty} \otimes \pi_{fin}$, where $\pi_{\infty}$ and $\pi_{fin}$ denote the archimedean and finite parts of $\pi$ respectively.\\

\section{Local systems}
\label{sec2}

Let $G$ denote the algebraic similitude symplectic group $GSp(4)$ over $\mathbb{Q}$. If $(\rho, V_{\rho})$ is an irreducible finite dimensional representation of $G$ with highest weight $\lambda$ and $\Gamma_{2}(N)$ acts on $\mathfrak{H}_{2} \times V_{\rho}$ as $\gamma.(\tau, v) = (\gamma\tau, \rho(\gamma)v)$ for any $\gamma \in \Gamma_{2}(N)$, $\tau \in \mathfrak{H}_{2}$, and $v \in V_{\rho}$, then the quotient space $\Gamma_{2}(N)\backslash (\mathfrak{H}_{2} \times V_{\rho})$ is a locally constant sheaf on $\mathcal{A}_{2, N}(\mathbb{C})$. It is denoted by $\mathbb{V}(\lambda)$. $\mathbb{V}(\lambda)$ is called the local system on $\mathcal{A}_{2, N}$ of the parameter $\lambda$. If $\overline{\mathcal{A}_{2, N}(\mathbb{C})}$ denotes a toroidal compatification of $\mathcal{A}_{2, N}(\mathbb{C})$, then the local system $\mathbb{V}(\lambda)$ has an extension on $\overline{\mathcal{A}_{2, N}(\mathbb{C})}$ denoted by $\overline{\mathbb{V}}(\lambda)$. \\

First of all, we need a convenient parametrization for weight $\lambda$. We begin with real points of $G$, $G(\mathbb{R}) = GSp(4, \mathbb{R})$. Let $Z_{G}$ be the center of $G$, which is given by scalar multiplications. Also $K_{G} = U(2).Z_{G}(\mathbb{R})$ in which $U(2)$ denotes the unitary group. We take $T$ as the product of the standard diagonal Cartan subgroup of $U(2)$ with $Z_{G}(\mathbb{R})$. If $\mathfrak{t}_{\mathbb{C}}$ denotes the complexification of the Lie algebra of $T$, then $\mathfrak{t}_{\mathbb{C}} \simeq \lbrace(x_{1}, x_{2}, z)\rbrace = \mathbb{C}^{3}$. We take $\varepsilon_{i}((x, z)) = x_{i}$, $i = 1, 2$ and $\nu((x, z)) = z$ as a basis for $\mathfrak{t}_{\mathbb{C}}^{\vee}$. Let $\alpha = \varepsilon_{1} - \varepsilon_{2}$ and $\beta = 2\varepsilon_{2}$. Then, $R_{c}^{+} = \lbrace \alpha \rbrace$ and $R_{n}^{+} = \lbrace \beta, \alpha + \beta, 2\alpha + \beta \rbrace$ be the  set of the positive compact roots and non-compact roots of the root system of $G(\mathbb{C}) = GSp(4, \mathbb{C})$ respectively. $\lambda$ is a highest weight, therefore it can be written as $\lambda = l\varepsilon_{1} + m\varepsilon_{2} + c\nu \in \mathfrak{t}_{\mathbb{C}}^{\vee} $ such that $l, m, c \in \mathbb{Z}$ and $l + m = c$ mod $2$. So $\lambda$ can be parametrized by a triple of integers $(l, m, c)$ with mentioned condition. This is the parametrization that will be used in this paper.\\

\noindent Let $\eta: G \to \mathbb{G}_{m}$ denote the multiplier homomorphism. Then the local system corresponded to $\eta^{c}$ is denoted by $\mathbb{C}(c)$ ($c \in \mathbb{Z}$). Also, $\mathbb{V}_{1, 0}$ denotes the local system defined by $\mathbb{V}_{1, 0} := R^{1}\pi_{*}(\mathbb{Q})$, where $R^{1}$ is the first direct image functor ($\mathbb{V}_{1, 0}$ is the local system corresponded to the $\eta^{-1}$ times the standard representation of $G(\mathbb{C})$ on $\mathbb{C}^{4}$.) Now $\mathbb{V}_{l, m}$ is the local system of weight $l + m$ and occurs `for first time' in $\text{Sym}^{l - m}(\mathbb{V}_{1,0})\otimes \text{Sym}^{m}(\wedge^{2}\mathbb{V}_{1, 0})$. Note that if $\lambda = (l, m, -l-m)$, then $\mathbb{V}(\lambda) = \mathbb{V}_{l, m}$. Therefore, according to the above setting, if $\lambda = (l, m , c)$ such that $l, m, c \in \mathbb{Z}$ and $l + m = c$ mod $2$, then $\mathbb{V}(\lambda)$ is the local system $\mathbb{V}_{l, m}$ twisted by $\mathbb{C}(c+l+m)$, i.e., $\mathbb{V}(\lambda) = \mathbb{V}_{l, m} \otimes \mathbb{C}(c+l+m)$. But we can always twist back $\mathbb{V}(\lambda)$ by $\mathbb{C}(-c-l-m)$. Although, our results do not depend on this normalization, but for simplicity in notation, we always assume that $ c = -l-m$. Moreover, we denote $\lambda$ simply by $(l, m)$, where $l, m \in \mathbb{Z}$.\\

Let $X_{N} = \mathcal{A}_{2, N}(\mathbb{C})$ and fix the local system $\mathbb{V}(\lambda)$ on $X_{N}$, where $\lambda = (l, m)$ such that $l, m\in \mathbb{Z}$. $H^{i}(X_{N}, \mathbb{V}(\lambda))$ and $H_{c}^{i}(X_{N}, \mathbb{V}(\lambda))$ denote the cohomology and compactly supported cohomology of $X_{N}$ respect to the locally constant sheaf $\mathbb{V}(\lambda)$. There is a natural map $f_{!}: H_{c}^{i}(X_{N}, \mathbb{V}(\lambda)) \to H^{i}(X_{N}, \mathbb{V}(\lambda)) $ whose image is called inner cohomology and denoted by $H_{!}^{i}(X_{N}, \mathbb{V}(\lambda))$. By work of Faltings and Chai, one knows that for a non-singular dominant weight $\lambda$, the cohomology group $H_{!}^{3}(X_{N}, \mathbb{V}(\lambda))$ is the only non-vanishing inner coholomlogy (see \cite{1} , Cor. to Thm. 7, p. 84 and \cite{2}, p. 233-7). Moreover, one knows that $H^{3}(X_{N}, \mathbb{V}(\lambda))$ (resp. $H_{c}^{3}(X_{N}, \mathbb{V})$) carries mixed Hodge structures of weight $\geq l + m +3$ (resp. $\leq l + m + 3$) and $H_{!}^{3}(X_{N}, \mathbb{V}(\lambda))$ carries a pure Hodge structure with Hodge filtration 

\begin{equation}
\label{1}
(0) \subset F^{3} \subset F^{2} \subset F^{1} \subset F^{0} = H_{!}^{3}(X_{N}, \mathbb{V}(\lambda)),  
\end{equation}
where $F^{r} = \bigoplus_{p \geq r} H_{!}^{p, q}(X_{N}, \mathbb{V}(\lambda)) $ (where $p, q \in \mathbb{Z}^{+}\cup \lbrace 0 \rbrace$ and $p + q = 3$).\\

\noindent The $(p, q)$-term can be described as the inner cohomology of $X_{N}$ with respect to a holomorphic vector bundle on $X_{N}$\footnote{These holomorphic vector bundles are constructed by restriction of equivariant holomorphic vector bundle on $\mathfrak{H}_{2}$. The explicit construction is mentioned in \cite{2}.}. Explicitly, according to \cite{2}, p. 237, we have the following table. In Table \ref{table1}, $H_{!}^{i}(\overline{X}_{N}, \overline{\mathcal{W}}(\Lambda))$ denotes the image of $\psi$, where $\psi$ is  
\begin{equation}
\label{inner}
\psi: H^{i}(\overline{X}_{N}, \overline{\mathcal{W}}(\Lambda)\otimes \mathcal{O}(-D)) \to H^{i}(\overline{X}_{N}, \overline{\mathcal{W}}(\Lambda)).
\end{equation}
Here $D$ denotes the boundary of the toroidal compactification of $X_{N}$.
\begin{table}[H]
\centering
\begin{tabular}{| l | l | l | l | l | l | }
    \hline
  
   Type & $(p, q)$ & $p= l(\omega_{i})$ & $q$ & $H_{!}^{p, q}(X_{N}, \mathbb{V}(\lambda))$ & $\Lambda$ \\ 
    \hline
 I & $(0, 3)$ & 0 & 3 & $ H_{!}^{3}(\overline{X}_{N}, \overline{\mathcal{W}}(-m, -l))$ & $(-m, -l)$  \\ 
    \hline
 II & $(1, 2)$ & 1 & 2 & $ H_{!}^{2}(\overline{X}_{N}, \overline{\mathcal{W}}(m + 2, - l))$ & $(m+2, -l)$   \\ 
   \hline
   III & $(2, 1)$ & 2 & 1 & $H_{!}^{1}(\overline{X}_{N}, \overline{\mathcal{W}}(l + 3, 1 - m))$ & $(l+3, 1-m)$   \\ 
    \hline
   IV & $(3, 0)$ & 3 & 0 & $ H_{!}^{0}(\overline{X}_{N}, \overline{\mathcal{W}}( l + 3,  m + 3))$ & $(l+3, m+3)$  \\ 
    \hline
   
    \end{tabular}
    \caption{$(p, q)$-term in Hodge filtration}
  \label{table1}
\end{table}

The inner cohomology space of $\overline{X}_{N}$ respect to a holomorphic bundle can be interpreted as a space of cuspidal Siegel modular forms of genus 2 over $\mathfrak{H}_{2}$ (see \cite{3}, p. 209). Therefore, the $(p, q)$-term in (\ref{1}) can be considered as a space of cuspidal Siegel modular forms of genus 2 of weight $\sigma_{\Lambda}$, where $\sigma_{\Lambda}$ is the finite dimensional representation of $GL(2, \mathbb{C})$ of highest weight $\Lambda$. For example,
\begin{equation}
\label{Faltings}
F^{3} \cong S_{l-m, m+3}(\Gamma_{2}(N)),
\end{equation} 
where  $S_{j, k}(\Gamma_{2}(N))$ is the space of Siegel modular forms corresponded to the representation $\text{Sym}^{j}\otimes \det^{k}$ of $GL(2, \mathbb{C})$ of highest weight $(j+k, k)$.\\

\section{Discrete series representations of $GSp(4, \mathbb{C})$ and cohomlogy classes}
\label{sec3}
 In this section, we will review some results of \cite{4} about the interpretation of automorphic representations of discrete types as cohomology classes. We are going to restrict ourself to the case $G$, which is the algebraic group $GSp(4)$ over $\mathbb{Q}$. This case was studied in \cite{5}. The main result in this section is a correspondence between the $(p, q)$-terms of Hodge structure of inner cohomology of $X_{N}$ respect to the local system $\mathbb{V}(\lambda)$ mentioned in Table \ref{table1} with four spaces of automorphic representations of $GSp(4, \mathbb{A})$ that are corresponded to the four discrete series representations of $G(\mathbb{R})^{+}$ with Harish-Chandra character $\chi_{\lambda + \rho}$.  \\
 
Let $X = \mathfrak{H}_{2}^{+} \bigsqcup \mathfrak{H}_{2}^{-} $, where $\mathfrak{H}_{2}^{+}$ denotes the Siegel upper half space of genus 2 and $\mathfrak{H}_{2}^{-}$ is the lower Siegel upper half space of genus 2. Now $(G, X)$ is a Shimura datum (see \cite{21}). Let $M_{G} = M(G, X)$ be the corresponding Shimura variety respect to Shimura datum $(G, X)$ which has a canonical model over its reflex field $E(G, X)$. If $K \subset G(\mathbb{A}_{f})$ is an open compact subgroup, then the complex points of the $K$-component of $M_{G}$ is
\begin{equation*}
 _{K}M_{G}(\mathbb{C}) = G(\mathbb{Q})\backslash (X \times G(\mathbb{A}_{f}))/K,
\end{equation*}
 and $M_{G}(\mathbb{C}) = \varprojlim_{K} \hspace{.001mm} _{K}M_{G}(\mathbb{C})$. Note that if $K(N)$ is defined as the open compact subgroup of $GSp(4, \mathbb{A}_{f})$ such that 
 \begin{equation}
 \label{level}
 K(N) = \lbrace k = \begin{pmatrix} A & B\\
                                C & D \end{pmatrix} \mid k \equiv I_{4} \quad \text{mod $N$} \rbrace, 
 \end{equation}
then $_{K(N)}M_{G}(\mathbb{C})$ is a finite disjoint union copies of  $X_{N}$. Therefore, $X_{N}$ is embedded into $K(N)$-component of $M_{G}(\mathbb{C})$.\\ 

Let $K_{G}$ be the subgroup of $G(\mathbb{R})$ defined by $U(2)Z_{G}(\mathbb{R})$. Let $(\sigma, V_{\sigma})$ be an irreducible finite dimensional representation of $K_{G}$ with highest weight $\Lambda$. We define an algebraic vector bundle 
\begin{equation*}
\mathcal{W}(\Lambda) := _{K}\mathcal{E}_{\Lambda} = G(\mathbb{Q}) \backslash (V_{\sigma} \times G(\mathbb{R}) \times G(\mathbb{A}_{f}))/K_{G}.K
\end{equation*}  
over $_{K}M_{G}(\mathbb{C})$. Given $K \subset G(\mathbb{A}_{f})$ neat in the sense of \cite{6}, there are good toroidal compactifications $_{K}\overline{M}_{G}$ which are projective rational over $E(G, X)$. For such $_{K}M_{G}$ there exists an extension of $\mathcal{W}(\Lambda)$ to a vector bundle over $_{K}\overline{M}_{G}$ denoted by $\overline{\mathcal{W}}(\Lambda)$. Note that over the complex points of  $_{K}M_{G}$, $\mathcal{W}(\Lambda)$ is a direct sum of finite copies of the holomorphic vector bundles considered previously in Section \ref{sec2}. Let $H_{!}^{i}(_{K}\overline{M}_{G}(\mathbb{C}), \overline{\mathcal{W}}(\Lambda))$ denote the inner cohomology of $_{K}\overline{M}_{G}(\mathbb{C})$ respect to $\overline{\mathcal{W}}(\Lambda)$ defined in (\ref{inner}). It is proved in \cite{4} that this cohomology spaces are independent of the choice of toroidal compactification. Let

\begin{equation*}
 H_{!}^{i}(M_{G}(\mathbb{C}), \overline{\mathcal{W}}(\Lambda)) =  \varinjlim_{K} H_{!}^{i}(_{K}\overline{M}_{G}(\mathbb{C}), \overline{\mathcal{W}}(\Lambda)).
\end{equation*}
It is clear that $H_{!}^{i}(\overline{X}_{N}, \overline{\mathcal{W}}(\Lambda)) \hookrightarrow H_{!}^{i}(M_{G}(\mathbb{C}), \overline{\mathcal{W}}(\Lambda))$ because $X_{N}$ was embedded into $M_{G}(\mathbb{C})$. Similarly, we can extend the local system $\mathbb{V}(\lambda)$ on $_{K}M_{G}(\mathbb{C})$ to a local system on $\overline{M}_{G}$ denoted by $\overline{\mathbb{V}}(\lambda)$. Let
  \begin{equation*}
 H_{!}^{i}(M_{G}(\mathbb{C}), \mathbb{V}(\lambda)) =  \varinjlim_{K} H_{!}^{i}(_{K}\overline{M}_{G}(\mathbb{C}), \overline{\mathbb{V}}(\lambda)).
\end{equation*}
Now $\lambda$ be a dominant weight. Because $_{K}M_{G}(\mathbb{C})$ is a finite disjoint union copies of $X_{N}$ and $H_{!}^{3}(X_{N}, \mathbb{V}(\lambda))$ is the only non-zero cohomology space, then $H_{!}^{3}(M_{G}(\mathbb{C}), \mathbb{V}(\lambda))$ is the only non-zero cohomology space (see \cite{1} , Cor. to Thm. 7, p. 84 and \cite{2}, p. 233-7). Moreover, it carries a Hodge structure of weight $3$ because each $H_{!}^{3}(_{K}\overline{M}_{G}(\mathbb{C}), \mathbb{V}(\lambda))$ 
has a Hodge structure of weight $3$. Finally, note that $H_{!}^{p, q}(X_{N}, \mathbb{V}(\lambda))$ is embedded into $ H_{!}^{p, q}(M_{G}(\mathbb{C}), \mathbb{V}(\lambda))$ because the Hodge structure is invariant under the direct limit.\\  

Let $\Lambda = (-m, -l, l+m)$, where $l, m \in \mathbb{Z}$, $l \geq m \geq 0 $. $\mathfrak{g}$ and $\mathfrak{k}$ denote the Lie algebras of $G(\mathbb{R})$ and $K_{G}$ respectively. Also, $\rho$ denotes the half sum of the positive roots. By the theory of $(\mathfrak{g}, K_{G})$-modules for the group $G(\mathbb{R})^{+}$ (see \cite{7}), for each $\Lambda$, there are four types of discrete series correspond to the four $\mathfrak{g}$-chambers in the positive $\mathfrak{k}$-chamber. This means that respect to our parametrization for $\Lambda$, we can parametrize these I-IV types of discrete series representations of $G(\mathbb{R})^{+}$ by Table \ref{table2}.

\begin{table}[H]
\centering
\begin{tabular}{ | l | l | l | l |  l | }
    \hline   
    Type & I & II & III & IV \\
    \hline
    $\Lambda$ & $(-m, -l, c)$ & $(m+2, -l, c)$ & $(l+3, 1-m, -c)$ & $(l+3, m+3, -c)$\\
    \hline
    $q_{\Lambda + \rho}$ & $3$ & $2$ & $1$ & $0$\\
    \hline
    $\Lambda + \rho$ & $(-m -1, -l - 2 , c)$ &  $(m+1, -l-2, c)$ & $(l+2, -m -1, -c)$ & $(l+2, m+1, -c)$\\
    \hline
    
     \end{tabular}
    \caption{Discrete series for $G(\mathbb{R})^{+}$}
  \label{table2}
\end{table}

\begin{remark}
Note that if $\pi_{\Lambda + \rho}$ is a discrete series representation of $G(\mathbb{R})^{+}$ \linebreak parametrized by $\Lambda + \rho = (a, b, c)$ (where $a+b=c$ mod 2), then its contragredient $\pi^{*}_{\Lambda + \rho}$ is parametrized by $(-b, -a, -c)$. 
The Table \ref{table2} is different from Table 2.2.1. in \cite{5}. Here we consider the contragredient of Table 2.2.2. in \cite{5} if $\lambda = (l, m)$. 
\end{remark}  

Now we are going to apply Theorem 1.2. in \cite{5} in the case $G = GSp(4)$. Fix $\lambda = (l, m)$ and the local system $\mathbb{V}_{l, m}$ on $\overline{M}_{G}$. Let $\sigma_{\Lambda}$ denote the irreducible finite dimensional representation of $K_{G}$ with highest weight $\Lambda$ for each type of $\Lambda$ in Table \ref{table2}. In \cite{5} (see p. 70), they show that if $l - m \geq 2$ and $m \geq 2$, then $\Lambda + \rho$'s in Table \ref{table2} are  sufficiently regular in the sense of Theorem 1.2. in \cite{5}. Therefore, if $l - m \geq 2$ and $m \geq 2$, Theorem 1.2. in \cite{5} implies that
\begin{equation}
\label{2}
H^{q_{\Lambda + \rho}}_{!}(M_{G}, \overline{\mathcal{W}}(\Lambda)) \cong Hom_{(\mathfrak{g}, K_{G})}(\pi_{\Lambda + \rho}^{*}, \mathcal{A}_{0}(G)), 
\end{equation}
where $q_{\Lambda + \rho} = \# \lbrace \alpha \in R_{n}^{-} \mid <\Lambda + \rho, \alpha> > 0 \rbrace$.\\

Let $(\pi, V_{\pi})$ be a cuspidal automorphic representation of $GSp(4, \mathbb{A})$ such that $V_{\pi} \subset \mathcal{A}_{0}(G)$. It is well-known that $\pi$ can be written as $\pi_{\infty} \otimes \pi_{fin}$, where $\pi_{\infty}$ is a $(\mathfrak{g}, K_{G})$-module of $G(\mathbb{R})^{+}$ and $\pi_{fin}$ is a $G(\mathbb{A}_{f})$-module. Also, $\mathcal{A}_{0}(G)$ can be decomposed into direct sum of irreducible cuspidal automorphic representations of $G(\mathbb{A})$. Therefore, the isomorphism in (\ref{2}) can be written as
\begin{equation}
\label{3}
H^{q_{\Lambda + \rho}}_{!}(M_{G}, \overline{\mathcal{W}}(\Lambda)) \cong \bigoplus_{\pi = \pi_{\infty}\otimes \pi_{fin}} m(\pi) \pi_{fin},
\end{equation}
where the sum runs over all irreducible cuspidal automorphic representations of $G(\mathbb{A})$ such that $\pi_{\infty}$ is a discrete series representation of type $\Lambda + \rho$.\\

\noindent Now it is clear that by comparing Table \ref{table1} and Table \ref{table2}, we have
\begin{proposition}
\label{p1}
Let $\lambda = (l, m)$ and $G = GSp(4)$. Fix the local system $\mathbb{V}_{l, m}$ on $M_{G}$. Suppose $l - m \geq 2$ and $l \geq 2$. Then, the $(p, q)$-terms of Hodge structure on $H_{!}^{3}(M_{G}, \mathbb{V}_{l, m})$ can be described under (\ref{3}) by
\begin{equation}
\label{4}
H_{!}^{p, q}(M_{G}, \mathbb{V}_{l, m}) = H_{!}^{q}(M_{G}, \overline{\mathcal{W}}(\Lambda)) \cong \bigoplus_{\pi = \pi_{\infty}\otimes \pi_{fin}} m(\pi)\pi_{fin}, 
\end{equation}
 where the sum runs over all irreducible cuspidal automorphic representations of $G(\mathbb{A})$ such that $\pi_{\infty}$ is a discrete series representation of $G(\mathbb{R})^{+}$ of type I-IV according to the type of $\Lambda$ in Table 1.\\
\end{proposition} 

The Proposition \ref{p1} gives a description of $(p, q)$-terms of Hodge structure on $H_{!}^{3}(M_{G}, \mathbb{V}_{l, m})$. Moreover, we know that $H_{!}^{p, q}(X_{N}, \mathbb{V}_{l, m}) \hookrightarrow H_{!}^{p, q}(M_{G}, \mathbb{V}_{l, m}) $ for any non-negative integers $p$ and $q$ such that $p + q = 3$. Also, $G(\mathbb{A}_{f})$ acts on $H_{!}^{3}(M_{G}, \mathbb{V}_{l, m})$ and we have $H_{!}^{3}(X_{N}, \mathbb{V}_{l, m}) = [H_{!}^{3}(M_{G}, \mathbb{V}_{l, m})]^{K(N)}$. This means that $H_{!}^{3}(X_{N}, \mathbb{V}_{l, m})$ is invariant part of $H_{!}^{3}(M_{G}, \mathbb{V}_{l, m})$ under the action of $K(N)$. It is known that the Hodge structure is invariant under the action of $G(\mathbb{A}_{f})$ (see \cite{8}). Therefore, under the isomorphism in (\ref{4}), we have 
\begin{equation}
\label{5}
H_{!}^{p, q}(X_{N}, \mathbb{V}_{l, m}) = H_{!}^{q}(\overline{X}_{N}, \overline{\mathcal{W}}(\Lambda)) \cong \bigoplus_{\pi = \pi_{\infty}\otimes \pi_{fin}} m(\pi)(\pi_{fin})^{K(N)}, 
\end{equation}
 where the sum runs over all irreducible cuspidal automorphic representations of $G(\mathbb{A})$ such that $\pi_{\infty}$ is a discrete series representation of $G(\mathbb{R})^{+}$ of the type I-IV according to the type of $\Lambda$ in Table 1. Here $(\pi_{fin})^{K(N)}$ denotes the invariant elements of $V_{fin} = \bigotimes_{\text{$\upsilon$ is finite place}} V_{\upsilon}$ under the compact open subgroup $K(N)$.\\

In summary, If we fix the local system $\mathbb{V}(\lambda)$ on the $X_{N}$, where $\lambda = (l, m)$ and $l - m \geq 2$ and $m \geq 2$, then the $(p, q)$-term of Hodge filtration of inner cohomology space $H_{!}^{3}(X_{N}, \mathbb{V}(\lambda))$ is isomorphic with  $ \bigoplus_{\pi = \pi_{\infty}\otimes \pi_{fin}} m(\pi)(\pi_{fin})^{K(N)}$ under (\ref{5}) and $\Lambda$ is determined by $q = q_{\Lambda + \rho}$ according to Table \ref{table2}.\\
 
\noindent Therefore, we are going to study the strict endoscopic part via the identification presented in (\ref{5}). Briefly, it means that we will construct the desired cohomological class in $H_{!}^{p, q}(X_{N}, \mathbb{V}_{l, m})$ in the form of $(\pi_{fin})^{K(N)}$ such that $\pi_{\infty}$ of $\pi = \pi_{\infty}\otimes \pi_{fin}$ is the discrete series representation of appropriate Type I-IV.\\

\begin{remark}
\label{r2}
We close this section by this comment that the Type I and IV are called the holomorphic and anti-holomorphic discrete series representations of $G(\mathbb{R})^{+}$ with infinitesimal character $\chi_{\lambda + \rho}$ and denoted by $\pi_{\Lambda + \rho}^{(3, 0)}$ and$\pi_{\Lambda + \rho}^{(0, 3)}$ respectively. Analogously, the Type II and III are called the non-holomorphic discrete series representations of $G(\mathbb{R})^{+}$ with infinitesimal character $\chi_{\lambda + \rho}$ and denoted by $\pi_{\Lambda + \rho}^{(1, 2)}$ and $\pi_{\Lambda + \rho}^{(2, 1)}$ respectively.  
\end{remark}

\section{Discrete series representations of $GO(2, 2)(\mathbb{R})$}
\label{sec4}
In this section we review the parametrization of discrete series representation for the algebraic group $GO(2, 2)$ over $\mathbb{R}$ stated in \cite{5}. We are going to denote $GO(2, 2)$ by $H$. This parametrization will be used later to make an explicit correspondence between the discrete series representations of $H(\mathbb{R})^{+}$ with discrete series representations of $G(\mathbb{R})^{+}$ under the local theta correspondence respect to the dual reductive pair $(G(\mathbb{R}), H(\mathbb{R}))$, where $G(\mathbb{R}) = GSp(4, \mathbb{R})$ and $H(\mathbb{R}) = GO(2, 2)(\mathbb{R})$. This enables us later to construct strictly endoscopic cohomology classes by global theta correspondence respect to the dual reductive pair $(G(\mathbb{A}), H(\mathbb{A}))$.\\

We will frequently need to look at representation of $SL(2, \mathbb{R})$. Let $J$ be the matrix $\begin{pmatrix} 0 & 1\\
                                                      -1 & 0

\end{pmatrix} \in \mathfrak{sl}(2, \mathbb{R})$ and $\mathfrak{st} = \mathbb{C}.J$ be the subalgebra generated by $J$ in $\mathfrak{sl}(2, \mathbb{C})$. For $k \in \mathbb{Z}$ with $k \geq 2$, let $\sigma_{k}$ (resp. $\sigma_{-k}$) be the representation of $\mathfrak{st}$ such that $\sigma_{k}(J) = k$ (resp. $\sigma_{-k}(J) = -k$). $\pi_{k}$(resp. $\pi_{-k}$) denotes the discrete series representation of $SL(2, \mathbb{R})$ with highest $\mathfrak{st}$-type $\sigma_{k}$(resp. $\sigma_{-k}$). Let $\mathfrak{gl}(2, \mathbb{R})$ denote the Lie algebra of $GL(2, \mathbb{R})$ and $K_{GL(2)} = U(1)Z_{GL(2)}(\mathbb{R})$ in which $Z_{GL(2)}$ is the center of $GL(2)$. Note that to parametrize the discrete series representations of $GL(2, \mathbb{R})$ we need an extra parameter $c$ for the central character. Therefore, a discrete series representation of $GL(2, \mathbb{R})$ (or a $(\mathfrak{gl(2)}, K_{GL(2)})$-module) can be parametrized by $\pi_{k}(c)$, where the restriction of $\pi_{k}(c)$ to $SL(2, \mathbb{R})$ is $\pi_{k}$ such that it has the central character $c$.\\

Assume $B$ is the totally split quaternion algebra $Mat_{22}(\mathbb{Q})$ with the quadratic form is given by $-\det$. Therefore, the similitude orthogonal group $GO(B)$ is isomorphic to $GO(2, 2)$ over $\mathbb{Q}$. The algebraic group $H = GO(2, 2)$ is disconnected. If $\nu$ denotes the similitude multiplier character for $GO(2, 2)$, then there is a map
\begin{equation*}
GO(2, 2) \to \lbrace \pm 1\rbrace, \qquad h \mapsto \frac{\det(h)}{\nu(h)^{2}}.
\end{equation*} 
The kernel of this map is denoted by $GSO(2, 2)$ which is the connected component of $GO(2, 2)$. Also, $GO(2, 2) = GSO(2, 2)\rtimes <t>$, where $t$ acts on $B$ by $t.b = b^{t}$. $<t>$ denotes the group of order $2$ generated by $t$. It is easy to check that there are exact sequences 
\begin{equation}
\label{6}
1 \to \mathbb{G}_{m} \to GL(2) \times GL(2) \rtimes <t>  \stackrel{\rho}{\to} GO(2, 2) \to 1  
\end{equation}
\begin{equation}
\label{7}
1 \to \mathbb{G}_{m} \to GL(2) \times GL(2) \stackrel{\rho}{\to} GSO(2, 2) \to 1,
\end{equation}
where $\rho$ is defined by $\rho(g_{1}, g_{2})b = g_{1}bg_{2}^{-1}$ and $\rho(t)b = b^{t}$ for all $g_{1}, g_{2} \in GL(2)$ and $b \in B$.\\

Consider $H = GO(2, 2)$ and $\tilde{H} = GL(2) \times GL(2)$ as algebraic groups over $\mathbb{Q}$. Let $Z_{H}$ be the center of $H$, i.e., the group of scalar multiplications on $B$. Also, let $K_{H} = Z_{H}(\mathbb{R}).(SO(2)\times SO(2))$ and $K_{\tilde{H}} = Z.SO(2) \times Z.SO(2)$, where $Z$ is the group of scalar matrices in $GL(2, \mathbb{R})$. Finally, $\mathfrak{h}$ and $\tilde{\mathfrak{h}}$ denote the Lie algebras of $H$ and $\tilde{H}$ respectively. Note that $\tilde{\mathfrak{h}} = \mathfrak{gl}(2, \mathbb{R}) \times \mathfrak{gl}(2, \mathbb{R})$.\\

\noindent It is clear that discrete series representations of $H(\mathbb{R})^{+}$ can be considered as $(\mathfrak{h}, K_{H})$-modules, and we are able to parametrize $(\mathfrak{h}, K_{H})$-modules in terms of their pullbacks to $(\tilde{\mathfrak{h}}, K_{\tilde{H}})$-modules. If $\mathfrak{k}_{\tilde{H}}$ denotes the Lie algebra of $K_{\tilde{H}}$, then we have 
\begin{equation*}
(\mathfrak{k}_{\tilde{H}})_{\mathbb{C}} = \lbrace k(t_{1}, z_{1}; t_{2}, z_{2}) = \begin{pmatrix}
z_{1} & t_{1} \\
-t_{1} & z_{1}
\end{pmatrix} \times \begin{pmatrix}
z_{2} & t_{2} \\
-t_{2} & z_{2}
\end{pmatrix} \mid z_{j}, t_{j} \in \mathbb{C}\rbrace.
\end{equation*}
Let $(a, c_{1}; b, c_{2})$ be the character of $(\mathfrak{k}_{\tilde{H}})_{\mathbb{C}}$ defined by 
\begin{equation*}
(a, c_{1}; b, c_{2})(k(t_{1}, z_{1}; t_{2}, z_{2})) = ia.t_{1} + ib.t_{2} + c_{1}.z_{1} + c_{2}.z_{2}.
\end{equation*}
The representations of $K_{\tilde{H}}$ of interest to us then correspond to parameters $a, b, c_{1},$ and $c_{2} \in \mathbb{Z}$ with $a \equiv c_{1}$ mod $2$ and $b \equiv c_{2}$ mod $2$. Moreover, such a representation is pullback from $K_{H}$ if and only if $c_{1} = -c_{2}$ mod $2$. It is because of the homomorphism $\rho: K_{\tilde{H}} \to K_{H}$ in (\ref{6}) that was given by $(z_{1}k_{1}, z_{2}k_{2}) \mapsto z_{1}z_{2}^{-1}(k_{1}^{-1}k_{2}, k_{1}k_{2})$. This means that we need only one parameter $c = c_{1}$ for the central character. Therefore, we denote the parameter of a representation which is the pullback from $K_{H}$ by $(a, b; c)$.\\

Since $\mathfrak{\tilde{h}} = \mathfrak{gl}(2, \mathbb{R}) \times \mathfrak{gl}(2, \mathbb{R})$ the parameters for discrete series $\pi$ for $H(\mathbb{R})^{+}$ are given by $\pi_{a}(c) \otimes \pi_{b}(c)$ where $a, b, c \in \mathbb{Z}$, $a, b \geq 2$, and $a \equiv b \equiv c$ mod $2$. For simplicity we denote $\pi$ by $\pi_{a, b}(c)$ which is the $(\mathfrak{h}, K_{H})$-module with parameter $(a, b; c)$.\\

\begin{remark}
\label{r3}
Let $\mathbb{K}_{H} = Z_{H}(\mathbb{R})(O(2) \times O(2))$ and
$\mathbb{K}_{\tilde{H}} = Z.O(2) \times Z.O(2)$. We consider $(\mathfrak{h}, \mathbb{K}_{H})$-modules via their pullback to $(\tilde{\mathfrak{h}}, \mathbb{K}_{\tilde{H}})$-modules. We denote by $\pi(a, b; c)$ the irreducible $(\mathfrak{h}, \mathbb{K}_{H})$-module whose pullback to $(\tilde{\mathfrak{h}}, \mathbb{K}_{\tilde{H}})$ yields 
\begin{equation}
\label{8}
\pi(a, b; c) = \begin{cases} \pi_{a, b}(c) + \pi_{b, a}(c) + \pi_{-a, -b}(c) + \pi_{-b, -a}(c) \quad \text{if $a \neq b$} \\
\pi_{a, b}(c) + \pi_{-a, -b}(c) \quad \text{if $a = b$ }
\end{cases}
\end{equation}
upon restriction to $(\tilde{\mathfrak{h}}, K_{\tilde{H}})$. 
\end{remark} 

\section{Global theta correspondence for the similitude dual reductive pair $(GO(2, 2), GSp(4))$}
\label{sec5}
In this section we review briefly the global theta correspondence for the similitude dual reductive pair $(GO(2, 2), GSp(4))$. This case was studied in \cite{5} and \cite{9}. Moreover, we are going to state some results on the non-vanishing of theta lifting in this case where stated in \cite{11} and \cite{10}. This section provides us a tool to construct strict endoscopic classes in the next section.\\

We consider $(W, < , >)$ (resp. $(V, ( , ))$) as a symplectic (resp. orthogonal) vector space over $\mathbb{Q}$ with $\dim_{\mathbb{Q}} W = 4$ (resp. $\dim_{\mathbb{Q}} V = 4$). We assume that $V = B$ i.e., the totally split quaternion algebra $Mat_{22}(\mathbb{Q})$ with the quadratic form is given by $-\det$. Let $G = GSp(W)$ and $H = GO(V)$ denote the similitude groups such that $\eta$ and $\nu$ are the similitude characters of $G$ and $H$ respectively. If $\mathbb{W} = W \otimes_{\mathbb{}} V$, then $\mathbb{W}$ is a symplectic vector space respect to form $\ll , \gg = < , >\otimes ( , )$. Therefore, $(G, H)$ forms a dual reductive pair in the similitude group $GSp(\mathbb{W})$. If $g \in G$ and $h \in H$, then we define $i(g, h) \in GSp(\mathbb{W})$, by $(v \otimes w). i(g, h) = h^{-1}v \otimes wg$. Note that $\eta(i(g, h) = \eta(g)\nu(h)^{-1}$. Let

\begin{equation*}
R = \lbrace (g, h) \in G \times H \mid \eta(g) = \nu(h)\rbrace.
\end{equation*} 

It is clear that there is a natural homomorphism $i: R \to Sp(\mathbb{W})$. Moreover, if $G_{1} = Sp(W)$ and $H_{1} = O(V)$, then $G_{1} \times H_{1} \subset R$.\\

Fix an additive character $\psi$ of $\mathbb{A}$, which is trivial on $\mathbb{Q}$. Let $\omega = \omega_{\psi}$ denote the Weil representation of $Mp(\mathbb{W})$, where $Mp(\mathbb{W})$ denotes the metaplectic cover of $Sp(\mathbb{W})$. We consider the model of $\omega$ (as defined in \cite{11}) on the Schwartz-Bruhat space $\mathcal{S}(V(\mathbb{A})^{2})$ of $V(\mathbb{A})^{2}$. The extended Weil representation of $R(\mathbb{A})$ for the global case is also defined in the same way. Namely, $R(\mathbb{A})$ acts on the space $\mathcal{S}(V(\mathbb{A})^{2})$ in such a way that the restriction of this action to $O(2, 2)(\mathbb{A}) \times Sp(4, \mathbb{A})$ is the Weil representation $\omega$ of $Mp(\mathbb{W})$ (See \cite{5}, p. 81). We call this representation of $R(\mathbb{A})$ the extended Weil representation of $R(\mathbb{A} )$, which we also denote it by $\omega$.\\

For $(g, h) \in R(\mathbb{A})$ and $\varphi \in \mathcal{S}(V(\mathbb{A})^{2})$, we define the theta kernel by

\begin{equation*}
\theta(g, h; \varphi) = \sum_{x \in V(\mathbb{Q})^{2}} \omega(g, h)\varphi(x).
\end{equation*}               

Let $f \in \mathcal{A}_{0}(H)$ and $\varphi \in \mathcal{S}(V(\mathbb{A})^{2})$. We can consider the integral
\begin{equation}
\label{9}
\theta(f; \varphi)(g) = \int_{H_{1}(\mathbb{Q}) \backslash H_{1}(\mathbb{A})}\theta(g, h_{1}h; \varphi)f(h_{1}h)dh_{1}, 
\end{equation}
where $h \in H(\mathbb{A})$ is any element such that $\eta(g) = \nu(h)$ and the measure $dh_{1}$ is given as in \cite{5}. It is a basic argument to show that the above integral is absolutely convergent and independent of choice of $h$.\\

Now this is the Lemma 5.1.9. in \cite{5}.

\begin{lemma}
\label{le1}
With the notation as above,
\begin{itemize}
\item[(i)] $\theta(f; \varphi)$ is left invariant under $\lbrace \gamma \in G(\mathbb{Q}) \mid \eta(\gamma) = \nu(\gamma^{'}) \quad \text{ for some} \quad \gamma^{'} \in H(\mathbb{Q})\rbrace$.
\item[(ii)] If the central character of $f$ is $\chi$, then

\begin{equation}
\label{10}
\theta(f; \varphi)(zg) = \chi\chi_{V}^{2}(z)\theta(f; \varphi)(g),
\end{equation}
where $\chi_{V}^{2}(z) = (z, \det V)$ is the quadratic character of $\mathbb{A}^{\times}$ associated to $V$.
\end{itemize}

\end{lemma}

\noindent The function $\theta(f; \varphi)$ can be considered as a function on $GSp(4, \mathbb{A})^{+} = \lbrace g \in G(\mathbb{A}) \mid \eta(g) \in \nu(H(\mathbb{A})) \rbrace$. By part (i) in the above lemma $\theta(f; \varphi)$ is invariant under $GSp(4, \mathbb{Q})^{+}$. Therefore $\theta(f; \varphi)$ is a function on $GSp(4, \mathbb{Q})^{+} \backslash GSp(4, \mathbb{A})^{+}$. We can extend this function to an automorphic form of $G(\mathbb{A})$ by insisting that it is left invariant under $G(\mathbb{Q})$ and zero outside the $G(\mathbb{Q})GSp(4, \mathbb{A})^{+}$. This means that we have a map $\theta$
\begin{equation}
\label{11}
\theta : \mathcal{A}_{0}(H) \to \mathcal{A}_{0}(G)
\end{equation}
\begin{align*}
f \longmapsto \theta(f; \varphi), \quad \text{where $\varphi \in \mathcal{S}(V(\mathbb{A})^{2})$}. 
\end{align*}
The above map is called the global theta correspondence for dual reductive pair $(H, G)$.\\

Now let $(\sigma, V_{\sigma})$ denote an irreducible cuspidal automorphic representation of $H(\mathbb{A})$ such that $V_{\sigma} \subset \mathcal{A}_{0}(H)$. We assume that $\sigma = \sigma_{\infty} \otimes \sigma_{fin}$. Let $\Theta(\sigma)$ be the space that is generated by the automorphic forms $\theta(f; \varphi)$ for all $f \in V_{\sigma}$ and   $\varphi \in \mathcal{S}(V(\mathbb{A})^{2})$. $G(\mathbb{A})$ acts on $\Theta(\sigma)$ by right translation. If $\Theta(\sigma)$ is the space of non-zero cusp forms, then its irreducible constituent provides a cuspidal representation of $G(\mathbb{A})$ denoted by $\Pi_{\sigma}$. If we write $\sigma = \otimes \sigma_{\upsilon}$ and $\Pi_{\sigma} = \otimes \Pi_{\upsilon}$ over places, then $\sigma_{\upsilon}^{*}$ is corresponded to $\Pi_{\upsilon}$ under the local theta correspondence for the dual reductive pair $(H, G)$ over the local field $\mathbb{Q}_{\upsilon}$. Explicitly, we can write $\Pi_{\sigma} = \otimes \theta(\sigma_{\upsilon}^{*}) = \otimes \theta(\sigma_{\upsilon})^{*}$, where $\theta$ denotes the local theta correspondence. This is a non-trivial result for the dual reductive pair $(H, G)$ is proved by S. Takeda in \cite{10}. This means that the global theta lift of $\sigma$ is determined by all of its local theta lifts.

\begin{remark}
\label{r4}
Note that in general if $\Theta(\sigma)$ is the spaces of non-zero cusp forms, then the statement $\Pi_{\sigma} = \otimes \theta(\sigma_{\upsilon}^{*}) = \otimes \theta(\sigma_{\upsilon})^{*}$ satisfies for a general dual reductive pair with assumption that $\sigma_{\upsilon}$ is tempered representation at $\upsilon \mid 2$. This is because of lack of How duality at $\upsilon \mid 2$. However, for the dual reductive pair $(GO(2, 2), GSp(4))$, there is no longer necessary to consider the temperedness assumption that is proved by S. Takeda in \cite{10}.
\end{remark} 

Since we are interested in constructing strict endoscopic classes, we state Theorem 5.2.1. in \cite{5} which provides us an explicit parametric description of the global theta lift from $H(\mathbb{A})$ to $G(\mathbb{A})$ in our case of interest.

\begin{theorem}
\label{th1}
Suppose $f \in \mathcal{A}_{0}(H)$ generates $(\mathfrak{h}, \mathbb{K}_{H})$-module isomorphic to $\pi_{H}^{*}$ with $\pi_{H} = \pi(l + m + 4, l - m + 2; c )$ (resp. $\pi(l + m + 4, m - l -2; c)$) with $l - m \geq 2$ and $m \geq 2$. Suppose that $\varphi \in \mathcal{S}(V(\mathbb{A})^{2})$ is such that $\theta(f; \varphi) \neq 0$. Then $\theta(f; \varphi) \in \mathcal{A}_{0}(G)$ and generates a $(\mathfrak{g}, K_{G})$-module isomorphic to the discrete series representation $\pi_{\Lambda + \rho}$, where $\Lambda = (m + 2, -l, c)$ (resp. $(l + 3, 1 - m, c)$).  
\end{theorem}

\noindent We can restate this theorem in terms of cuspidal automorphic representation. Let $(\sigma, V_{\sigma})$ be an irreducible cuspidal automorphic representation of $H(\mathbb{A})$ such that $V_{\sigma} \subset \mathcal{A}_{0}(H)$. We assume $\sigma$ is written (under isomorphism) as $\sigma_{\infty} \otimes \sigma_{fin}$. If $f \in V_{\sigma}$ is a non-zero element, then $\Theta(\sigma)$ is generated by $\theta(f; \varphi)$ for all $\varphi \in \mathcal{S}(V(\mathbb{A})^{2})$. We assume that $\Theta(\sigma) \neq 0$ and $\Pi_{\sigma}$ denotes the irreducible constituent of $\Theta(\sigma)$ as before. The Theorem \ref{th1} says that if $\sigma_{\infty}$ is isomorphic to $\pi_{H}^{*}$ with  $\pi_{H} = \pi(l + m + 4, l - m + 2; c )$ (resp. $\pi(l + m + 4, m - l -2; c)$) with $l - m \geq 2$ and $m \geq 2$, then $\Theta(\sigma)$ is a space of non-zero cusp forms since $\theta(f; \varphi)$ are cusp forms. Moreover, $\Pi_{\sigma} = \otimes \theta(\sigma_{\upsilon}^{*}) = \otimes \theta(\sigma_{\upsilon})^{*}$ such that $\Pi_{\sigma, \infty}$ is isomorphic to the discrete series representation $\pi_{\Lambda + \rho}$, where $\Lambda = (m + 2, -l, c)$ (resp. $(l + 3, 1 - m, c)$). Note that these are the non-holomorphic discrete series representations of $GSp(4, \mathbb{R})^{+}$ denoted by $\pi_{\Lambda + \rho}^{(1, 2)}$ and $\pi_{\Lambda + \rho}^{(2, 1)}$. We summarize the above discussion as the following corollary.
\begin{corollary}
\label{c1}
Let $(\sigma, V_{\sigma})$ be an an irreducible cuspidal atuomorphic representation of $H(\mathbb{A})$ such that $V_{\sigma} \subset \mathcal{A}_{0}(H)$. Moreover, $\sigma_{\infty}$ is isomorphic to $\pi_{H}^{*}$ with  $\pi_{H} = \pi(l + m + 4, l - m + 2; c )$ (resp. $\pi(l + m + 4, m - l -2; c)$) with $l - m \geq 2$ and $m \geq 2$, where $\sigma = \sigma_{\infty} \otimes \sigma_{fin}$. Assume that $\Theta(\sigma) \neq 0$. Then, $\Pi_{\sigma}$ is an irreducible cuspidal automorphic representation of $G(\mathbb{A})$ such that $\Pi_{\sigma, \infty}$ is isomorphic to the discrete series representation $\pi_{\Lambda + \rho}^{(1, 2)}$ (resp.  $\pi_{\Lambda + \rho}^{(2, 1)}$), where $\Pi_{\sigma} = \Pi_{\sigma, \infty} \otimes \Pi_{\sigma, fin}$. 
\end{corollary}
Note that we assume $\Theta(\sigma) \neq 0$. The non-vanishing of $\Theta(\sigma)$ is a hard problem for a general dual reductive pair. However, in our case of interest $(H, G)$, we have an explicit answer to the non-vanishing problem. To state the answer we need first to describe the irreducible admissible representations of $H$ over a local field $\mathbb{Q}_{\upsilon}$.\\

Let assume $G = GSp(4)$ and $H = GO(2, 2)$. Also $H^{\circ}$ denotes the subgroup $GSO(2, 2)$ of $H$. Note that $H = H^{\circ} \rtimes <t>$, according to (\ref{6}). We have an isomorphism $c: H^{\circ} \to H^{\circ}$ defined by $c(h) = tht$. Let $(\pi, V_{\pi})$ be an irreducible admissible representation of $H^{\circ}$ over the local field $\mathbb{Q}_{\upsilon}$. We define $\pi^{c}$ by taking $V_{\pi^{c}} = V_{\pi}$ and by letting $\pi^{c}(h)(v) = \pi(c(h))(v)$ for all $h \in H^{\circ}$ and $v \in V_{\pi}$. Then we have 
\begin{itemize}
\item If $\pi \ncong \pi^{c}$, then $\text{Ind}_{H^{\circ}}^{H}$ is irreducible, and we denote it by $\pi^{+}$.
\item If $\pi \cong \pi^{c}$, then $\text{Ind}_{H^{\circ}}^{H}$ is reducible. Moreover, $\text{Ind}_{H^{\circ}}^{H} = \pi^{+} \oplus \pi^{-}$ where $\pi^{\pm}$ are irreducible admissible representations of $H$. Explicitly, $\pi^{+}\mid_{H^{\circ}} \cong \pi^{-}\mid_{H^{\circ}} \cong \pi$ and $t$ acts on $\pi^{\pm}$ via a linear operator $\theta^{\pm}$ with the property that $(\theta^{\pm})^{2} = \text{Id}$ and $\theta^{\pm} \circ h = c(h) \circ \theta^{\pm}$ for any $h \in H^{\circ}$. 
\end{itemize}

\begin{remark}
\label{r5}
 The exact sequence (\ref{7}) over the local field $\mathbb{Q}_{\upsilon}$ is 
\begin{equation*}
1 \to \mathbb{Q}_{\upsilon}^{\times} \to GL(2, \mathbb{Q}_{\upsilon}) \times GL(2, \mathbb{Q}_{\upsilon}) \stackrel{\rho}{\to} GSO(2, 2)(\mathbb{Q}_{\upsilon}) \to 1.
\end{equation*}
Therefore, it is clear that an irreducible admissible representation $\pi$ of $H^{\circ}(\mathbb{Q}_{\upsilon})$, via pullback to $GL(2, \mathbb{Q}_{\upsilon}) \times GL(2, \mathbb{Q}_{\upsilon})$, can be considered as $\tau_{1} \otimes \tau_{2}$, where $\tau_{i}$ is an irreducible admissible representation of $GL(2, \mathbb{Q}_{\upsilon})$ and $\omega_{\tau_{1}}\omega_{\tau_{2}} = 1$. Here $\omega_{\tau_{i}}$ denotes the central character of $\tau_{i}$. Since the above argument, we denote the admissible representation $\pi$ of $H^{\circ}(\mathbb{Q}_{\upsilon})$ by $\pi = \pi(\tau_{1}, \tau_{2})$, where $\omega_{\tau_{1}}\omega_{\tau_{2}} = 1$. Moreover, if $\pi = \pi(\tau_{1}, \tau_{2})$, then $\pi^{c} = \pi(\tau_{2}^{\vee}, \tau_{1}^{\vee})$. Here $\tau^{\vee}$ denotes the representation of $GL(2, \mathbb{Q}_{\upsilon})$ such that $V_{\tau^{\vee}} = V_{\tau}$ and $\tau^{\vee}(g)v = \tau((g^{-1})^{t})v$ for all $g \in GL(2, \mathbb{Q}_{\upsilon})$ and $v \in V_{\tau}$.
\end{remark}

Now we describe the global cuspidal automorphic representations of $H(\mathbb{A})$ in terms of global cuspidal automorphic representations of $H^{\circ}(\mathbb{A})$. Note that if $\pi$ is an irreducible cuspidal automorphic representation of $H^{\circ}(\mathbb{A})$, then  we define $\pi^{c}$ by taking $V_{\pi^{c}} = V_{\pi}$ and by letting $\pi^{c}(h)(v) = \pi(c(h))(v)$ for all $h \in H^{\circ}(\mathbb{A})$ and $v \in V_{\pi}$.\\

\noindent Let $\sigma$ be an irreducible cuspidal automorphic representation of $H(\mathbb{A})$. We consider the restriction of $\sigma$ to $H^{\circ}(\mathbb{A})$ denoted by $\sigma \mid_{H^{\circ}(\mathbb{A})}$. Therefore, $V_{\sigma \mid_{H^{\circ}(\mathbb{A})}} = \lbrace f \mid_{H^{\circ}(\mathbb{A})} \mid f \in V_{\sigma} \rbrace$. We are going to denote the $\sigma \mid_{H^{\circ}(\mathbb{A})}$ by $\sigma^{\circ}$. Following \cite{12} (Lemma 2, p. 381), we have
\begin{lemma}
\label{le2}
If $\sigma$ is an irreducible cuspidal automorphic representation of $H(\mathbb{A})$, then there is an irreducible cuspidal automorphic representation $\pi$ of $H^{\circ}(\mathbb{A})$ such that either $\sigma^{\circ} = \pi$ with $\pi = \pi^{c} $ or $\sigma^{\circ} = \pi \oplus \pi^{c}$. 
\end{lemma}
Thus we obtain a map from cuspidal automorphic representations of $H(\mathbb{A})$ to cuspidal automorphic representations of $H^{\circ}(\mathbb{A})$ modulo the action of $\lbrace 1, c \rbrace $.\\

\begin{remark}
\label{r6}
In \cite{12} the authors consider the case that $\pi$ is a cuspidal representation of $GO(1, 3)(\mathbb{A})$. However, exactly a same argument will prove the results if $\pi$ is a cuspidal automorphic representation of either $GO(4)(\mathbb{A})$ or $GO(2, 2)(\mathbb{A})$.
\end{remark}

Now we want to find the fibre above a fix cuspidal atuomorphic representation $\pi$ of $H^{\circ}(\mathbb{A})$. We define $\tilde{\pi}$ be sum of the cuspidal automorphic representations of $H(\mathbb{A})$ lying above $\pi$ under the above map. This means that $\tilde{\pi} = \oplus_{i} \sigma_{i}$ such that $\sigma_{i}$ is a cuspidal automorphic representation of $H(\mathbb{A})$ which either $\sigma_{i}^{\circ} = \pi$ (if $\pi = \pi^{c}$), or $\pi \oplus \pi^{c}$ (if $\pi \neq \pi^{c}$).\\

\noindent There are two different cases: 

\begin{itemize}
\item[(i)] $\pi \ncong \pi^{c}$ (See \cite{12}, p. 382-383). 
\begin{proposition}
\label{p2}
Assume $\pi \ncong \pi^{c}$. Then
\begin{equation}
\label{13}
\tilde{\pi} = \bigoplus_{\delta} \otimes_{\upsilon} \pi_{\upsilon}^{\delta(\upsilon)},
\end{equation}
where $\delta$ runs over all the maps from the set of all places of $\mathbb{Q}$ to $\lbrace \pm \rbrace$ with the property that $\delta(\upsilon) = + $ for almost all places of $\mathbb{Q}$, and $\delta(\upsilon) = +$ if $\pi_{\upsilon} \ncong \pi_{\upsilon}^{c}$. Moreover each $\otimes_{\upsilon} \pi_{\upsilon}^{\delta(\upsilon)}$ is isomorphic to a cuspidal automorphic representation of $H(\mathbb{A})$.
\end{proposition}

\item[(ii)] $\pi \cong \pi^{c}$ (See \cite{12}, p. 382-383).\\
In this case we can not stay the same result as Proposition \ref{p2}. Basically, the fibre over a cuspidal automorphic representation $\pi$ of $H^{\circ}(\mathbb{A})$ is not known when $\pi = \pi^{c}$. However, if $\pi$ is generic the we have 
\begin{proposition}
\label{p3}
Assume $\pi$ is a generic  cuspidal automorphic representation of $H(\mathbb{A})$, and $\pi \cong \pi^{c}$. Then
\begin{equation}
\label{14}
\tilde{\pi} = \bigoplus_{\delta} \otimes_{\upsilon} \pi_{\upsilon}^{\delta(\upsilon)},
\end{equation}
where $\delta$ runs over all the maps from the set of all places of $\mathbb{Q}$ to $\lbrace \pm  \rbrace$ with the property that $\delta(\upsilon) = + $ for almost all places of $\mathbb{Q}$, and $\delta(\upsilon) = + $ if $\pi_{\upsilon} \ncong \pi_{\upsilon}^{c}$, and $\prod_{\upsilon} \delta(\upsilon) = + $. Moreover, each $\otimes_{\upsilon} \pi_{\upsilon}^{\delta(\upsilon)}$ is isomorphic to a cuspidal automorphic representation of $H(\mathbb{A})$.   
\end{proposition}
 
\end{itemize}

\begin{remark}
\label{r7}
 The exact sequence (\ref{7}) over $\mathbb{A}$ is 
\begin{equation*}
1 \to \mathbb{A}^{\times} \to GL(2, \mathbb{A}) \times GL(2, \mathbb{A}) \stackrel{\rho}{\to} GSO(2, 2)(\mathbb{A}) \to 1.
\end{equation*}
Therefore, it is clear that an irreducible cuspidal automorphic representation $\pi$ of $H^{\circ}(\mathbb{A})$, via pullback to $GL(2, \mathbb{A}) \times GL(2, \mathbb{A})$, can be considered as $\tau_{1} \otimes \tau_{2}$, where $\tau_{i}$ is an irreducible cuspidal automorphic representation of $GL(2, \mathbb{A})$ and $\omega_{\tau_{1}}\omega_{\tau_{2}} = 1$. Again $\omega_{\tau_{i}}$ denotes the central character of $\tau_{i}$. Since the above argument we denote a cuspidal automorphic  representation $\pi$ of $H^{\circ}(\mathbb{A})$ by $\pi = \pi(\tau_{1}, \tau_{2})$, where $\omega_{\tau_{1}}\omega_{\tau_{2}} = 1$. Moreover, if $\pi = \pi(\tau_{1}, \tau_{2})$, then $\pi^{c} = \pi(\tau_{2}^{\vee}, \tau_{1}^{\vee})$. Here $\tau^{\vee}$ denotes the cuspidal automorphic representation of $GL(2, \mathbb{A})$ such that $V_{\tau^{\vee}} = V_{\tau}$ and $\tau^{\vee}(g)v = \tau((g^{-1})^{\vee})v$ for all $g \in GL(2, \mathbb{A})$ and $v \in V_{\tau}$. An irreducible cuspidal automorphic representation $\sigma$ of $H(\mathbb{A})$ is denoted by $(\pi, \delta)$, where $\pi = \pi(\tau_{1}, \tau_{2})$ is a cuspidal automorphic representation of $H^{\circ}(\mathbb{A})$ and $\delta$ is an acceptable choice as in the Propositions \ref{p2} or \ref{p3}. $\sigma = (\pi, \delta)$ is called an extension of $\pi$ to $H(\mathbb{A})$.
\end{remark}
\begin{remark}
\label{r8}
We should mention that if $\pi = \pi(\tau_{1}, \tau_{2})$ is a cuspidal automorphic representation of $H^{\circ}(\mathbb{A})$, then $\pi$ is generic if and only if both of $\tau_{i}$'s are infinite dimensional. This is because of the classical result that infinite dimensional cuspidal automorphic representations of $GL(2, \mathbb{A})$ are the only ones that be generic. We always assume that $\pi = \pi(\tau_{1}, \tau_{2})$ is generic. 
\end{remark}

Now we are at the position to state the non-vanishing criteria of the global theta lifting for our desired dual reductive pair $(H, G)$, where $H = GO(2, 2)$ and $G = GSp(4)$. The following result, proved by S. Takeda in \cite{10}, shows that the global non-vanishing for dual reductive pair $(H, G)$ can be detected by non-vanishing of local theta lifting at all local places (archimedean and non-archimedean).

\begin{theorem}[S. Takeda]
\label{th2}
Let $\sigma = (\pi, \delta)$ be an unitary cuspidal automorphic representation of $H(\mathbb{A})$ such that $\pi = \pi(\tau_{1}, \tau_{2})$ and $\tau_{i}$'s are infinite dimensional. Then the global theta lifting of $\sigma$ (i.e., $\Theta(\sigma)$) to $G(\mathbb{A})$ is non-zero if and only if each local constituent $\sigma_{\upsilon}$ has a non-zero lift to $G(\mathbb{Q}_{\upsilon})$.
\end{theorem}

\noindent Theorem \ref{th2} means that we need only to have a local criteria for non-vanishing. This is also done by S. Takeda in our desired case $(H, G)$ (see \cite{10}, p. 19).
\begin{theorem}
\label{th3}
Let $\mathbb{Q}_{\upsilon}$ be a local field ($\upsilon$ is not necessarily non-archimedean). Let $\sigma_{\upsilon}$ be an irreducible admissible representation of $H(\mathbb{Q}_{\upsilon})$ such that $\sigma_{\upsilon}^{c} = \sigma_{\upsilon}$. Then only $\sigma_{\upsilon}^{+}$ has non-zero theta lift to $G(\mathbb{Q}_{\upsilon})$.    
\end{theorem}

Finally we have the following theorem stated in \cite{13}.
\begin{theorem}
\label{th4}
Let $\pi = \pi(\tau_{1}, \tau_{2})$ denote a generic cuspidal automorphic representation of $H^{\circ}(\mathbb{A})$. Then there exists a cuspidal automorphic representation $\sigma = (\sigma, \delta)$ of $H(\mathbb{A})$ such that the global theta lift to $G(\mathbb{A})$ is non-zero (i.e., $\Theta(\sigma) \neq 0$) and $\Theta(\sigma)$ is generic.    
\end{theorem}
\begin{remark}
\label{r9}
The existence of a non-vanishing global theta correspondence for the orthogonal dual reductive pair $(O(2, 2), Sp(4))$ is a classical result by R. Howe and I. I. Piatetski-Shapiro in \cite{11}. Moreover, they prove that the $\Theta(\sigma)$ admits Whittaker model if $\sigma$ is generic. S. Takeda shows the same result for the similitude case. 
\end{remark}

Now we can summarize the above results in our desired way. 

\begin{corollary}
\label{c2}
Let $\pi = \pi(\tau_{1}, \tau_{2})$ be a generic cuspidal automorphic representation of $H^{\circ}(\mathbb{A})$. Let $\sigma^{+}$ denote the extension of $\pi$ to $H(\mathbb{A})$ such that $\sigma^{+} = (\pi, \delta)$ and $\delta = +$ for all local places (archimedean and non-archimedean). Then $\Theta(\sigma^{+}) \neq 0$ and $\Theta(\sigma^{+})$ is generic. Moreover, $\sigma^{+}$ is the only extension of $\pi$ to $H(\mathbb{A})$ which has non-zero global theta lift to $G(\mathbb{A})$.

\end{corollary}
\begin{proof}
Fix a generic cuspidal automorphic representation $\pi = \pi(\tau_{1}, \tau_{2})$ of $H^{\circ}(\mathbb{A})$. According to Theorem \ref{th4} we have an extension $\sigma = (\pi, \delta)$ such that $\sigma$ is a cuspidal automorphic representation of $H(\mathbb{A})$ and $\Theta(\sigma) \neq 0$. According to the Theorem \ref{th2} this implies that all local theta lifts to $G(\mathbb{Q}_{\upsilon})$ are non-zero for any local place $\upsilon$. Therefore, Theorem \ref{th3} implies  that $\sigma_{\upsilon} = \pi_{\upsilon}^{+}$ because otherwise it is local theta lift should be zero. This means that $\delta = +$ for all places. Moreover, it clear that if $\sigma^{'} = (\pi, \delta^{'})$ such that $\delta(\upsilon) = -$ for at least one place $\upsilon$, then the global theta lift to $G(\mathbb{A})$ is zero because the local theta lift at place $\upsilon$ is zero.
\renewcommand{\qedsymbol}{}
\end{proof}

\section{The construction of strict endoscopic cohomological classes}
\label{sec6}

In this section we are going to use Corollaries \ref{c1} and \ref{c2} to construct strict endoscopic classes. Let $G = GSp(4)$ and $H = GO(2, 2)$ be algebraic groups over $\mathbb{Q}$. Moreover $\eta$ and $\nu$ denote the similitude characters of $G$ and $H$ receptively. Also, $H^{\circ} = GSO(2, 2)$ be the subgroup of $G$.\\

Fix $\lambda = (l, m)$ (Note that we consider the convention $c = -l-m$.) Assume that $l - m \geq 2$ and $m \geq 2$. Fix the local system $\mathbb{V}(\lambda)$ on $M_{G}(\mathbb{C})$. Our goal is constructing strict endoscopic cohomology classes of $H_{!}^{3}(M_{G}, \mathbb{V}(\lambda))$. According to the isomorphism in (\ref{4}), first we need to construct irreducible cuspidal automorphic representation $\pi = \pi_{\infty} \otimes \pi_{fin}$ of $GSp(4, \mathbb{A})$ such that $\pi_{\infty}$ be isomorphic with $\pi_{\Lambda + \rho}^{(1, 2 )}$ or $\pi_{\Lambda + \rho}^{(2, 1)}$ for a $\Lambda$ of Type II or III in Table \ref{table2}. Secondly, we need to show that the finite part of $\pi$, $\pi_{fin}$, does not appear as a finite part of any irreducible cuspidal automorphic representation of $GSp(4, \mathbb{A})$ such that its archimedean component be isomorphic with  $\pi_{\Lambda + \rho}^{(3, 0)}$ or$\pi_{\Lambda + \rho}^{(0, 3)}$ for a $\Lambda$ of Type I or VI in Table \ref{table2}. Therefore under the isomorphism (\ref{4}), the $\pi_{fin}$ is a strict endoscopic cohomolgy class because it contributors in $H_{!}^{1, 2}(M_{G}, \mathbb{V}(\lambda)) \oplus H_{!}^{2, 1}(M_{G}, \mathbb{V}(\lambda))$ (i.e, non-holomorphic terms) but does not contributes in $H_{!}^{3, 0}(M_{G}, \mathbb{V}(\lambda)) \oplus H_{!}^{0, 3}(M_{G}, \mathbb{V}(\lambda))$ (i.e, holomorphic terms). \\

We start with a fix irreducible cuspidal automorphic representation $\pi(\tau_{1}, \tau_{2})$ of $H^{\circ}(\mathbb{A})$. We assume that $\tau_{i}$'s are infinite dimensional, i.e., $\pi(\tau_{1}, \tau_{2})$ is generic. We assume $\pi(\tau_{1}, \tau_{2}) \cong \pi_{\infty} \otimes \pi_{fin}$ such that $\pi_{\infty}^{*}$ (as a $(\mathfrak{h}, K_{H})$-module) be isomorphic with either $\pi_{l + m + 4, l - m + 2}(c)$ or $\pi_{l + m + 4, m - l -2}(c)$ where $c = -l - m$. Now $\sigma^{+}$ is the extension of $\pi(\tau_{1}, \tau_{2})$ on $H(\mathbb{A})$. This means that $\sigma^{+} = (\pi(\tau_{1}, \tau_{2}), \delta)$ such that $\delta = +$ for all local places. If we consider $\sigma^{+} \cong \sigma^{+}_{\infty} \otimes \sigma^{+}_{fin}$, then according to Remark \ref{r3}, $\sigma^{+}_{\infty}$ (as a $(\mathfrak{h}, \mathbb{K}_{H})$-module) is isomorphic with either $\pi(l + m + 4, l - m + 2; c)$ or $\pi(l + m + 4, m - l -2; c)$.\\

According to Corollary \ref{c2} the global theta lift of $\sigma^{+}$(i.e., $\Theta(\sigma^{+})$) is non-zero and $\Theta(\sigma^{+})$ is a generic cuspidal automorphic representation of $G(\mathbb{A})$. Now according to Corollary \ref{c1}, $\Pi_{\sigma^{+}}$ (i.e, the irreducible constituent of $\Theta(\sigma^{+})$) is an irreducible cuspidal automorphic representation of $G(\mathbb{A})$ such that its archimedean constituent $\Pi_{\sigma^{+}, \infty}$ is isomorphic either to the discrete series representation $\pi_{\Lambda + \rho}^{(1, 2 )}$ or $\pi_{\Lambda + \rho}^{(2, 1)}$, where we assume $\Pi_{\sigma^{+}} \cong \Pi_{\sigma^{+}, \infty} \otimes \Pi_{\sigma^{+}, fin}$ such that $\Lambda = (m + 2, -l, c)$ for the case $\pi_{\Lambda + \rho}^{(1, 2 )}$ and $\Lambda = (l + 3, -m, -c)$ for the case $\pi_{\Lambda + \rho}^{(2, 1)}$.\\

\noindent Now $\Pi_{\sigma^{+}, fin}$ is a non-holomorphic cohomology class of $H_{!}^{3}(M_{G}, \mathbb{V}(\lambda))$ under the isomorphism described in (\ref{4}). Explicitly, $\Pi_{\sigma^{+}, fin}$ contributes in $H_{!}^{1, 2}(M_{G}, \mathbb{C})$ or $H_{!}^{2, 1}(M_{G}, \mathbb{C})$ under the isomorphism in (\ref{4}). Our next goal is to show that $\Pi_{\sigma^{+}, fin}$ is strict endoscopic class, i.e., $\Pi_{\sigma^{+}, fin}$  does not contribute in  $H_{!}^{3, 0}(M_{G}, \mathbb{C})$ and $H_{!}^{0, 3}(M_{G}, \mathbb{C})$ under the isomorphism in (\ref{4}). It sufficient to  prove that $\Pi_{\sigma^{+}, fin}$ can not be appear as a finite part of any irreducible cuspidal automorphic representation $\Pi^{'} \cong \Pi^{'}_{\infty} \otimes \Pi_{\sigma^{+}, fin}$ of $GSp(4, \mathbb{A})$ such that its irreducible archimedean component, $\Pi^{'}_{\infty}$, is isomorphic with either to the discrete series representation $\pi_{\Lambda + \rho}^{(3, 0)}$ or $\pi_{\Lambda + \rho}^{(0, 3)}$, where $\Lambda = (-m, -l, c)$ in the case $\pi_{\Lambda + \rho}^{(3, 0)}$ and $\Lambda = (l + 3, m + 3, -c)$ in the case $\pi_{\Lambda + \rho}^{(0, 3)}$.\\

To prove the above claim we need the formula in (\ref{f}) which is a part of a main result by R. Weissauer in \cite{14}. Before proving this claim, we need to state the formula (\ref{f}). First we have the following definition as in \cite{11}.

\begin{definition}
\label{def1}
 If $\pi = \otimes_{\upsilon} \pi_{\upsilon}$ and $\pi^{'} = \otimes_{\upsilon} \pi_{\upsilon}^{'}$ are cuspidal automorphic representations of $GSp(4, \mathbb{A})$, then we say $\pi$ and $\pi^{'}$ are weakly equivalent if $\pi_{\upsilon} \cong \pi_{\upsilon}^{'} $ for almost all places.
\end{definition}

\begin{remark}
\label{r10}
 Note that we don't have the strong multiplicity one theorem for cuspidal automorphic representations of $GSp(4, \mathbb{A})$. It means that the weakly equivalent does not imply that $\pi$ is isomorphic to $\pi^{'}$. However, if $\pi$ and $\pi^{'}$ are generic, then they satisfy the strong multiplicity one for $GSp(4, \mathbb{A})$.
\end{remark}

Now we are following \cite{14} (Theorem 5.2, p. 186). Let $\pi$ be an irreducible cuspidal automorphic representation of $G(\mathbb{A})$ such that $\pi_{fin}$ contributes in $H^{3}_{!}(M_{G}, \mathbb{V}(\lambda))$. Moreover, assume that $\pi_{fin}$ is constructed as a finite part of  $\Pi_{\sigma^{+}}$ such that $\sigma^{+}$ is the positive extension of a representation of $GSO(2, 2)(\mathbb{A})$. This means that $\pi$ is weak endoscopic lift in terms of \cite{14}. If $\pi^{'} = \otimes \pi_{\upsilon}^{'}$ is weakly equivalent to $\pi$, then the multiplicity of $\pi^{'}$ in the discrete spectrum is equal to
\begin{equation}
\label{f} m(\pi^{'}) = \frac{1}{2}(1 + (-1)^{e(\pi^{'})}),
\end{equation}
where $e(\pi^{'})$ denotes the finite number of representations $\pi_{\upsilon}^{'}$ which do not have local Whittaker model.\\

Now we can state the main result in this section.
\begin{theorem}
\label{th5}
Fix $\lambda = (l, m)$ (Note that we consider the convention $c = -l-m$). Assume that $l - m \geq 2$ and $m \geq 2$. We consider the fixed local system $\mathbb{V}(\lambda)$ on $M_{G}(\mathbb{C})$. Let $\pi(\tau_{1}, \tau_{2})$ be an irreducible cuspidal automorphic representation $\pi(\tau_{1}, \tau_{2})$ of $H^{\circ}(\mathbb{A})$. We assume that $\tau_{i}$'s are infinite dimensional, i.e., $\pi(\tau_{1}, \tau_{2})$ is generic. We assume $\pi(\tau_{1}, \tau_{2}) \cong \pi_{\infty} \otimes \pi_{fin}$ such that $\pi_{\infty}^{*}$ (as a $(\mathfrak{h}, K_{H})$-module) be isomorphic with either $\pi_{l + m + 4, l - m + 2}(c)$ or $\pi_{l + m + 4, m - l -2}(c)$. Then $\Pi_{\sigma^{+}, fin}$ which is constructed above is an strict endoscopic class of $H_{!}^{3}(M_{G}, \mathbb{V}(\lambda))$  under the isomorphism in (\ref{4}).
\end{theorem} 

\begin{proof}
We already show that $\Pi_{\sigma^{+}, fin}$ contributes in the non-holomorphic terms of the Hodge structure of $H_{!}^{3}(M_{G}, \mathbb{V}(\lambda))$. We need only to prove our claim which says that  $\Pi_{\sigma^{+}, fin}$ can not be appear as a finite part of any irreducible cuspidal automorphic representation $\Pi^{'} \cong \Pi^{'}_{\infty} \otimes \Pi_{\sigma^{+}, fin}$ of $GSp(4, \mathbb{A})$ such that its irreducible archimedean component $\Pi^{'}_{\infty}$ be isomorphic with either the discrete series representation $\pi_{\Lambda + \rho}^{(3, 0)}$ or $\pi_{\Lambda + \rho}^{(0, 3)}$. To prove this let $\Pi^{'}$ be isomorphic with $\pi_{\lambda}^{(3, 0 )} \otimes \Pi_{\sigma^{+}, fin}$ or $\pi_{\lambda}^{(0, 3)} \otimes \Pi_{\sigma^{+}, fin}$. It is clear that $\Pi^{'}$ is weakly equivalence with $\Pi_{\sigma^{+}}$ because they have same finite parts. Moreover, Corollary \ref{c2} implies that $\Pi_{\sigma^{+}}$ is generic as the irreducible constituent of $\Theta(\sigma^{+})$. Because $\Pi_{\sigma^{+}}$ is generic, then all of its local non-archimedean components $\Pi_{\sigma^{+}, \upsilon}$ are generic, i.e., admit local Whittaker models. This means that at $\Pi^{'}$ all of its finite components have local Whittaker models. But we assume that the infinite part of $\Pi^{'}$ is $\pi_{\Lambda + \rho}^{(3, 0)}$ or $\pi_{\Lambda + \rho}^{(0, 3)}$. It is well-known that these two discrete series do not admit local Whittaker model. Therefore $e(\Pi^{'})$ in the formula (\ref{f}) is one. This implies that $m(\Pi^{'}) = 0$ which means that there are not cuspidal automorphic representations of the forms $\pi_{\lambda}^{(3, 0 )} \otimes \Pi_{\sigma^{+}, fin}$ or $\pi_{\lambda}^{(0, 3)} \otimes \Pi_{\sigma^{+}, fin}$. Therefore, $\Pi_{\sigma^{+}, fin}$ is strict endoscopic.
\renewcommand{\qedsymbol}{}     
\end{proof}

The above construction has a nice description in terms of modular forms. By the classical theory of modular forms of $GL(2)$ (\cite{15}), we know that there is a one to one correspondence between irreducible infinite dimensional cuspidal automorphic representations of $GL(2, \mathbb{A})$ and newforms (i.e., normalized Hecke eigenforms) on the upper half plane $\mathfrak{H}_{1}$. Therefore the fixed $\pi(\tau_{1}, \tau_{2})$ (as an irreducible cuspidal automorphic representation $\pi(\tau_{1}, \tau_{2})$ of $H^{\circ}(\mathbb{A})$) can be corresponded to a pair of newforms $(f_{1}, f_{2})$. Now if we assume that $f_{1}$ has the weight $l + m + 4 $ and $f_{2}$ has the weight either $l - m + 2$ or $m - l - 2$, then $\pi(\tau_{1}, \tau_{2})$ will satisfies the conditions in Theorem \ref{th5}. Therefore the above construction can be considered as a map

\begin{equation}
\label{15}
\Theta: (f_{1}, f_{2}) \longleftrightarrow \pi(\tau_{1}, \tau_{2}) \longmapsto \Pi_{\sigma^{+}, fin}.
\end{equation}  
  
This provides an injective map between the space of pairs of newforms of $GL(2)$ into the strict endoscopic part of $H_{!}^{3}(M_{G}, \mathbb{V}(\lambda))$ which is denoted by $H^{3}_{\text{End}^{\text{s}}}(M_{G}, \mathbb{V}(\lambda))$. The injectivity of the above map is a consequence of Corollary \ref{c2} that states $\sigma^{+}$ is the only non-zero global theta lift to $GSp(4, \mathbb{A})$. Note that for the injectivity of $\Theta$ we should assume that $f_{2}$ has the weight either $l - m + 2$ or $m - l -2$. This is possible that $\Pi_{\sigma^{+}, fin}$  appears as the finite part of two cuspidal automorphic representations with infinite constituents isomorphic with $\pi_{\Lambda + \rho}^{(3, 0)}$ and $\pi_{\Lambda + \rho}^{(0, 3)}$ respectively. Therefore $\Pi_{\sigma^{+}, fin}$ could contribute in $H_{!}^{1, 2}(M_{G}, \mathbb{V}_{l, m})$ and $H_{!}^{2, 1}(M_{G}, \mathbb{V}_{l, m})$ simultaneously. \\

\noindent However, our main problem was the study the strict endoscopic part of \linebreak $H_{!}^{3}(X_{N}, \mathbb{V}_{l, m})$ which is denoted by    
$H^{3}_{\text{End}^{\text{s}}}(X_{N}, \mathbb{V}_{l, m})$. Although Theorem \ref{th5} provides non-zero strict endoscopic classes, but we lose our control on the level structure. Explicitly, if we want to construct a non-zero strict endoscopic of $H_{!}^{3}(X_{N}, \mathbb{V}_{l, m})$, under the isomorphism in (\ref{5}), we have to find the smallest positive integer $N$ such that $(\Pi_{\sigma^{+}, fin}) ^{K(N)} \neq 0$. Then, $(\Pi_{\sigma^{+}, fin})^{K(N)}$ contributes non-trivially in the strict endoscopic part of $H_{!}^{3}(X_{N}, \mathbb{V}_{l, m})$. We are going to call the positive integer $N$ the level of the cuspidal automorphic representation
$\Pi_{\sigma^{+}}$. Note that this problem occurs because of using Shimura varieties method. This means that we lose our control on the level groups because of taking direct limit over all different levels.\\

\noindent This difficulty is solvable if we have an explicit way to calculate the level of the cuspidal automorphic representation
$\Pi_{\sigma^{+}}$ according to input data in the map $\Theta$ in (\ref{15}). This means that we are going to assume that $f_{i}$'s have conductor $a(f_{i})$ in the sense of classical modular forms of $GL(2)$ (\cite{18}). This is equivalent to assume $\sigma = \pi(\tau_{1}, \tau_{2})$ has conductor $(a(\tau_{1}), a(\tau_{2}))$. Now we want to now that what is the level of $\Pi_{\sigma^{+}}$?. If we determine the level of  $\Pi_{\sigma^{+}}$, then $(\Pi_{\sigma^{+}, fin})^{K(N)}$ contributes in the strict endoscopic part of $H_{!}^{3}(X_{N}, \mathbb{V}_{l, m})$. This is the next main result in this paper where will be studied in the next section. 

\section{Local theory and level structure}
\label{sec7}

This section contains the answer of the question mentioned in the previous section. We assume that $\pi(\tau_{1}, \tau_{2})$ is a generic cuspidal automorphic representation of $H^{\circ}(\mathbb{A})$. Also $\Pi_{\sigma^{+}}$ is the irreducible constituent of $\Theta(\sigma^{+})$ where constructed in Section \ref{sec6}. According to Section \ref{sec5}, we have $\Pi_{\sigma^{+}} = \otimes \theta((\sigma^{+}_{\upsilon})^{*}) = \otimes \theta(\sigma^{+}_{\upsilon})^{*} $. This means that $\Pi_{\sigma^{+}}$ is completely determined by local theta correspondence for dual reductive pair $(H, G)$ over all local fields $\mathbb{Q}_{\upsilon}$. This is a central statement in this section. Because if we are able to define local levels such that they are related to the global level $N$, then we can calculate $N$ according to local data.\\

\noindent Explicitly, we need to calculate the level of the cuspidal automorphic representation $\Pi_{\sigma^{+}}$. We know that $\Pi_{\sigma^{+}} = \bigotimes_{\upsilon} \Pi_{\sigma^{+}, \upsilon}$. According to theory of local newforms of $GSp(4)$ (\cite{16}), we can define the local level $N_{\upsilon}$ for each irreducible admissible representation $\Pi_{\sigma^{+}, \upsilon}$ over $\mathbb{Q}_{\upsilon}$ such that almost all $N_{\upsilon}$ are 1 and $N = \Pi_{\upsilon} N_{\upsilon}$. Therefore, If we are able to calculate the local level $N_{\upsilon}$ for every place, then we are able to get the desired answer in the global case. This means that we should answer the following local analogous question. If $\sigma^{+}_{\upsilon}$ has level $(a(\tau_{1, \upsilon}), a(\tau_{2, \upsilon}))$ at the finite place $\upsilon$ then what is the level of $\theta((\tau_{1, \upsilon}, \tau_{2, \upsilon}))$?.\\

To answer this question for admissible representations over local fields, we have the following difficulties:

\begin{itemize}

\item[(i)] First of all, we need to review the generalization of theory of newforms for the admissible representations of $GSp(4, \mathbb{Q}_{\upsilon})$, where $\upsilon$ is a non-archimedean local place. This means that we need to have a well-defined definition for conductor of an irreducible admissible representation of $G(\mathbb{Q}_{\upsilon})$. For this problem we are going to use the theory of paramodular vectors based on the work done by B. Roberts and R. Schmidt (\cite{16}). They present the definition of paramodular representations which are the generalization of local newforms of $GL(2)$ for the reductive algebraic group $ G = GSp(4)$.\\

\item[(ii)] Secondly, we need a precise correspondence between the admissible representations of $(GL(2, \mathbb{Q}_{\upsilon}) \times GL(2, \mathbb{Q}_{\upsilon})) / \mathbb{Q}_{\upsilon}^{\times}$ to admissible representations of the $GSp(4, \mathbb{Q}_{\upsilon})$ under the local theta correspondence. This part was done by W. Gan and S. Takeda (\cite{17}).

\item[(iii)] Finally we need a precise matching between the above results to find a proper answer for finding the level of $\theta((\tau_{1, \upsilon}, \tau_{2, \upsilon}))$, where $(\tau_{1, \upsilon}, \tau_{2, \upsilon})$ is an admissible representation of $(GL(2, \mathbb{Q}_{\upsilon}) \times GL(2, \mathbb{Q}_{\upsilon}))/ \mathbb{Q}_{\upsilon}^{\times}$ over the non-archimedean local place $\upsilon$ of the level $(a(\tau_{1, \upsilon}), a(\tau_{2, \upsilon}))$.        

\end{itemize}

\subsection{Paramodular admissible representations of $GSp(4, \mathbb{Q_{\upsilon}})$}
Here we are going to review briefly the theory of local newforms of $GSp(4)$ presented by B. Roberts and R. Schmidt in \cite{16}. Let $F = \mathbb{Q}_{\upsilon}$ be a non-archimedean local field. The ring of integer of $F$ will be denoted by $\mathcal{O}$, and $\mathcal{P}$ is its maximal ideal\footnote{Note that because we are working on the field $\mathbb{Q}$, so the finite places are prime numbers and the generator of $\mathcal{P}$ is the corresponded prime number.}. Let $q = \# \mathcal{O}/ \mathcal{P}$ be the cardinality of the residue field. We let $\varpi$ denote a fixed generator of $\mathcal{P}$. The normalized valuation $\vartheta$ on $F$ has the property that $\vartheta(\varpi) = 1$. Moreover, for the normalized absolute value we have $\nu(\varpi) = \mid \varpi \mid = q^{-1}$.\\

For $m \geq 1$ let $K_{2}(m)$ be the subgroup of $GL(2, \mathcal{O})$ consisting of matrices of the form
\begin{equation}
\begin{pmatrix}
a & b \\
c & d  \end{pmatrix}, \quad \text{ $c \equiv 0$ (mod $\mathcal{P}^{m}$)}, \quad d \in 1 + \mathcal{P}^{m}. 
\end{equation}

\noindent Classically, we can state the following well-known theorem of Casselman (\cite{19}).

\begin{theorem}
\label{th6}
Let $(\tau, V)$ be an infinite dimensional irreducible admissible representation of $GL(2, F)$ and $\psi$ is a fixed additive character of $F$. Let $V^{(m)}$ denote the space of $K_{2}(m)$ fixed vectors in $V$. Then:
\begin{itemize}
\item[(i)] There exists an $m \geq 0$ such that $V^{(m)} \neq 0$. Let $a(\tau)$ denote the least non-negative integer with this property. Then the $\varepsilon$-factor of $\tau$ has the form 
\begin{equation*}
\varepsilon(s, \tau, \psi) = cq^{(2c(\psi) - c(\tau))s}
\end{equation*}
for a constant $c \in \mathbb{C}^{*}$ which is independent of $s$.

\item[(ii)] For all $m \geq a(\tau)$ we have $\dim(V^{(m)})$ = $m - a(\tau) + 1$. 

\end{itemize} 
\end{theorem}   

\begin{remark}
\label{r11}
The positive integer $a(\tau)$ is called the conductor (level) of $\tau$. Moreover, if $a(\tau) = 0$ the $\tau$ is called unramified representation (Spherical representation). Note that $\dim(V^{(a(\tau))})$ = $1$. Also, a non-zero element in the one dimensional space $V^{a(\tau)}$ is called a local newform. A very good reference for details on this topic is \cite{18} by R. Schmidt.
\end{remark}

Table \ref{table3} presents conductors of irreducible admissible representations of $GL(2, F)$ defined in Theorem \ref{th6}. Note that $a(\chi)$ denotes the conductor of the character $\chi$ and $\St_{GL(2)}$ is the Steinberg representation of $GL(2)$.\\

 \begin{table}[H]
 \centering
\begin{tabular}{ | l | l | } 
    \hline
    representation $\tau$ & $a(\tau)$\\
    \hline
    $\tau(\chi_{1}, \chi_{2})$, $\chi_{1}\chi_{2}^{-1} \neq \mid \mid^{\pm 1}$ & $a(\chi_{1}) + a(\chi_{2})$ \\
    \hline
    $\chi \St_{GL(2)}$, $\chi$ unramified & $1$ \\
    \hline
    $\chi \St_{GL(2)}$, $\chi$ ramified & $2a(\chi)$ \\
    \hline
    $\tau$ is supercuspidal & $a(\tau) \geq 2$ \\
    \hline
    \end{tabular}
    \caption{Conductor of irreducible admissible representations of $GL(2, F)$}
    \label{table3}
    \end{table}

Now let $\tau = \tau_{\infty} \otimes \tau_{fin}$ be an automorphic representation of $GL(2, \mathbb{A})$. It is well-known that $\tau_{\upsilon}$'s are unramified for almost all places. Let $S$ denotes the set of finite places such that $\tau_{\upsilon}$'s are ramified. The we can consider $a(\tau) = \prod_{\upsilon \in S} \mathcal{P}^{a(\tau_{\upsilon})}$. $a(\tau)$ is called the level of the automorphic representation $\tau$. We can replace the maximal ideal $\mathcal{P}$ with the corresponding prime number $p$ as its generator. In this case $a(\tau)$ can be written as positive integer $ N_{\tau} = \prod_{\upsilon \in S} p^{a(\tau_{\upsilon})}$. With this concept it is clear that we want to fix a generic cuspidal automorphic representation $\pi(\tau_{1}, \tau_{2}) \cong \tau_{1} \otimes \tau_{2}$ of $H^{\circ}(F)$ of the level $(a(\tau_{1}), a(\tau_{2}))$ and find the level of its theta lift $\theta(\tau_{1}, \tau_{2})$ to $G(F)$, where $F = \mathbb{Q}_{\upsilon}$ is a non-archimedean local field. But we need to have a well-defined definition of level of an admissible representation of $G(F)$.\\

Now we are going to state the concept of level for an admissible representation of $G(F) = GSp(4, F)$ presented by B. Roberts and R. Schmidt in \cite{16}. Fix the local non-archimedean field $F = \mathbb{Q}_{\upsilon}$. For any non-negative integer $n$ we define $K(\mathcal{P}^{n})$ to be the subgroup of all $k$ in $GSp(4, F)$ such that $\det(k) \in \mathcal{O}^{\times}$ and 
\begin{equation}
\label{16}
k \in \begin{pmatrix} \mathcal{O} &\mathcal{O} & \mathcal{O} & \mathcal{P}^{-n}\\ 
                 \mathcal{P}^{n}  & \mathcal{O} & \mathcal{O} & \mathcal{O}\\
                  \mathcal{P}^{n} & \mathcal{O}  & \mathcal{O} & \mathcal{O} \\
                  \mathcal{P}^{n}  & \mathcal{P}^{n} & \mathcal{P}^{n}  & \mathcal{O}    
                                  \end{pmatrix}.
\end{equation}
$K(\mathcal{P}^{n})$ is called the paramodular group of level $\mathcal{P}^{n}$. Now we have the following definition.\\

\begin{definition}
\label{def2}
Let $(\pi, V)$ be an admissible representation of $G(F)$, where $F = \mathbb{Q}_{\upsilon}$ is a non-archimedean local field and $\pi$ has trivial central character. For a non-negative integer $n \geq 0$, let $V(n)$ denote the subspace of $V$ consisting of all vectors $v \in V$ such that $\pi(k)v = v$ for all $k \in K(\mathcal{P}^{n})$. A non-zero vector in $V(n)$ is called a paramodular vector of level $n$. Moreover, if there is a non-negative integer $n$ such that $V(n) \neq 0$, then $\pi$ is called a paramodular representation. If $\pi$ is a paramodular representation, we denote its minimal paramodular level of $\pi$ by $a(\pi)$ and it is called level of $\pi$. 
\end{definition}

First of all if $\pi$ is a paramodular representation of level $0$, it means that $V(0) \neq 0$. But $K(\mathcal{P}^{0}) = GSp(4, \mathcal{O})$. This means that $\pi$ is an unramified representation (Spherical representation). Moreover, there are representations of $GSp(4, F)$ such that they are not paramodular. The complete list of these representations can be found in \cite{16} (p. 121).\\

The first main result is Theorem 7.5.1. (Uniqueness at the Minimal Level) in \cite{16}. This theorem says:

\begin{theorem}[B. Roberts and R. Schmidt]
\label{th7}
Let $(\pi, V)$ be an irreducible admissible representation of $GSp(4, F)$ with trivial central character. Assume the $\pi$ is paramodular, and let $a(\pi)$ be the level of $\pi$. Then $\dim(V(a(\pi))) = 1$.

\end{theorem}

Theorem \ref{th7} shows that the paramodular representations of $GSp(4, F)$ at minimal paramodular level have the same property like the infinite dimensional representations of $GL(2, F)$. This provides us an appropriate candidates for newforms for $GSp(4, F)$. We should mention it again that all the admissible representations of $GSp(4, F)$ are not paramodular. However, the Theorem 7.5.4. (Generic Main Theorem) in \cite{16} shows that generic representations of $GSp(4, F)$ are paramodular.

\begin{theorem}[B. Roberts and R. Schmidt]
\label{th8} 
Let $(\pi, V)$ be a generic, irreducible, admissible representation of $GSp(4, F)$ with trivial central character. Then $\pi$ is paramodular.
\end{theorem}
\begin{remark}
\label{r12}
We know that the local constituents of $\Pi_{\sigma^{+}, fin}$ are generic. Therefore they are paramodular representations. Therefore, there is a well-defined level concept to calculate the level of $\Pi_{\sigma^{+}, \upsilon} = \theta((\tau_{1, \upsilon}, \tau_{2, \upsilon}))^{*}$ based on $a(\tau_{i, \upsilon})$. 
\end{remark}

\begin{remark}
\label{r13}
Note that B. Roberts and R. Schmidt assume that our representation of $GSp(4, F)$ has trivial central character. This is only because of simplicity. The central character can be any unramified character. It is clear that after a twist we can assume that the representation has trivial central character. We are going to assume the central character is trivial. We want to apply their result on the local theta lift $\theta((\tau_{1}, \tau_{2}))$, where $\tau_{i}$'s are irreducible admissible representations of $GL(2, F)$ where $\omega(\tau_{1}) \omega(\tau_{2}) = 1$. Note that the central character of $\pi = \pi(\tau_{1}, \tau_{2})$ is $\omega(\tau_{1}) = \omega(\tau_{2})^{-1}$. But if we want to get trivial central character for $\theta((\tau_{1}, \tau_{2}))$, we have to assume $\omega(\tau_{1}) = \omega(\tau_{2})$ be trivial (according to Lemma \ref{le1}). This is our assumption in the rest of this paper. It means that we consider pairs of cuspidal automorphic representations of $H^{\circ}(\mathbb{A})$, $\pi = \pi(\tau_{1}, \tau_{2})$, such that both of $\tau_{i}$'s have trivial central characters. 
\end{remark}

\noindent There is a classification of all the irreducible admissible representations of $GSp(4, F)$. They can be divided into two different groups: Non-supercuspidal representations (Induced representations) and supercuspidal representations.\\

\noindent Non-supercuspidal representations are constructed as subrepresentations or quotient representations of induced representations. Any admissible representation of $GSp(4, F)$ which is not non-supercupidal is called supercuspidal. $GSp(4, F)$ has three conjugacy classes of parabolic subgroups denoted by $B$, $P$, and $Q$. The following list describes all the parabolic induction from these parabolic subgroups. 

\begin{itemize}
\item \textbf{Parabolic Induction from $B$}

Let $\chi_{1}$, $\chi_{2}$, and $\sigma$ be characters of $F^{\times}$. Consider the character of $B(F)$ given by 
\begin{equation*}
\begin{pmatrix} a & * & * & *\\ 
                             & b & * & *\\
                             &  & \eta a^{-1} & * \\
                             & & & \eta b^{-1}    
                                          \end{pmatrix} \longmapsto \chi_{1}(a)\chi_{2}(b)\sigma(\eta).
\end{equation*}
Let $(\chi_{1} \times \chi_{2} \rtimes \sigma, V)$ denote the standard induced representation of $G(F)$ given by the above character of $B(F)$. $V$ consists of all locally constant functions $f: G(F) \to \mathbb{C}$ with the transformation property 
\begin{equation*}
f(hg) = \mid a^{2}b\mid \mid \eta \mid^{-3/2} \chi_{1}(a)\chi_{2}(d)\sigma(\eta)f(g), 
\end{equation*}

for all $h = \begin{pmatrix} a & * & * & *\\ 
                             & b & * & *\\
                             &  & \eta a^{-1} & * \\
                             & & & \eta b^{-1}    
                                          \end{pmatrix} \in B(F) $. 
$G(F)$ acts on $V$ via right translation. Note that the central character of $\chi_{1} \times \chi_{2} \rtimes \sigma$ is $\chi_{1}\chi_{2}\sigma^{2}$.\\

\item \textbf{Parabolic Induction from $P$}\\ 

Let $(\pi , V_{\pi})$ be an admissible representation of $GL(2, F)$ and $\sigma$ be a character of $F^{\times}$. Then we can consider a representation of $P(F)$ given by
\begin{equation*}
\begin{pmatrix}
A & * \\
 & \eta A^{'}
\end{pmatrix} \longmapsto \sigma(\eta)\pi(A),
\end{equation*}
where $A \in GL(2, F)$ and $A^{'} = \begin{pmatrix} 0 & 1 \\
                                                    1 & 0 \end{pmatrix} (A^{t})^{-1}                                
                                                    \begin{pmatrix} 0 & 1 \\ 
                                                    1 & 0 \end{pmatrix}$.\\
Let $(\pi \rtimes \sigma, V)$ denote the standard induced representation of $G(F)$ given by above representation of $P(F)$. $V$ consists of all locally constant functions $f: G(F) \to \mathbb{C}$ with the transformation property 
\begin{equation*}
f(hg) = \mid \det(A)\eta \mid^{3/2}\sigma(\eta)\pi(A)f(g), 
\end{equation*}

for all $h = \begin{pmatrix}
A & * \\
 & \eta A^{'}
\end{pmatrix} \in P(F) $. $G(F)$ acts on $V$ via right translation.. Note that the central character of $\pi \rtimes \sigma$ is $\omega_{\pi}\sigma^{2}$.\\

\item \textbf{Parabolic Induction from $Q$}\\ 

Let $(\pi , V_{\pi})$ be an admissible representation of $GL(2, F)$ and $\chi$ be a character of $F^{\times}$. Then we can consider a representation of $Q(F)$ given by
\begin{equation*}
\begin{pmatrix}
t & * & * & * \\
 & a & b & * \\
 & c & d & * \\
 & & & \eta t^{-1}
\end{pmatrix} \longmapsto \chi(t)\pi(\begin{pmatrix} a & b\\
                                                          c & d \end{pmatrix}).
\end{equation*}

Let $(\chi \rtimes \pi, V)$ denote the standard induced representation of $G(F)$ given by above representation of $Q(F)$. $V$ consists of all locally constant functions $f: G(F) \to \mathbb{C}$ with the transformation property 
\begin{equation*}
f(hg) = \mid t^{2}(ad - cd)^{-1} \mid \chi(t)\pi(\begin{pmatrix} a & b\\
                                                          c & d \end{pmatrix})f(g), 
\end{equation*}

for all $h = \begin{pmatrix}
t & * & * & * \\
 & a & b & * \\
 & c & d & * \\
 & & & \eta t^{-1}
\end{pmatrix} \in Q(F) $. $G(F)$ acts on $V$ via right translation. Note that the central character of $\chi \rtimes \pi$ is $\omega_{\pi}\chi$.\\                                            
                                                     
\end{itemize}

Let $(\pi, V)$ be an irreducible admissible representation of $GSp(4, F)$. If $\pi$ is isomorphic to an irreducible subrepresentation or an irreducible quotient representation of a parabolically induced representation of one of the parabolic subgroups $B$, $P$, or $Q$, then $\pi$ is called non-supercuspidal representation. Jr. Sally  and M. Tadic classified all irreducible non-supercuspidal representations of $GSp(4, F)$ in \cite{20}. This classification is presented as Table A.1. in \cite{16} (p. 270). Moreover B. Roberts and R. Schmidt determine all irreducible admissible paramodular representations of $GSp(4, F)$. They provide Table A.12. in \cite{16} (p. 291) which gives us complete calculation of level of paramodular representations of $GSp(4, F)$. We are going to use these tables later in this section to calculate the level of $\theta((\tau_{1, \upsilon}, \tau_{2, \upsilon}))$ as a paramdular representation.\\

Now let $\pi = \pi_{\infty} \otimes \pi_{fin}$ be an automorphic representation of $GSp(4, \mathbb{A})$. It is well-known that $\pi_{\upsilon}$'s are unramified for almost all places. Also, assume that the finite local constituent $\pi_{\upsilon}$ is paramodular representation if $\upsilon \in S$, where $S$ denotes the set of finite places such that $\pi_{\upsilon}$'s are ramified. Let \linebreak $a(\pi) = \prod_{\upsilon \in S} \mathcal{P}^{a(\pi_{\upsilon})}$. $a(\pi)$ is called the level of the automorphic representation $\pi$. Note we can replace the maximal ideal $\mathcal{P}$ with the corresponding prime number $p$ as its generator. In this case the level can be written as $ N_{\pi} = \prod_{\upsilon \in S} p^{a(\pi_{\upsilon})}$.

\subsection{Local theta correspondence for $(GSO(2, 2)(F), GSp(4, F))$}
We are going to study the second difficulty in this part. Consider the dual reductive pair $(GSO(2, 2)(F), GSp(4, F))$ over $F$, where $F = \mathbb{Q}_{\upsilon}$ is a non-archimedean local field. We know that $GSO(2, 2)(F)$ is isomorphic to \linebreak $(GL(2, F) \times GL(2, F))/ \mathbb{G}_{m, F}$. Therefore, any irreducible admissible representation of $GSO(2, 2)(F)$ is isomorphic with a pair $\pi(\tau_{1}, \tau_{2})$, where $\tau_{i}$'s are irreducible admissible representations of $GL(2, F)$ such that $\omega_{\tau_{1}}\omega_{\tau_{2}} = 1$. We need to figure out precisely that $\theta((\tau_{1}, \tau_{2}))$ is isomorphic with which one of irreducible admissible representations of $GSp(4, F)$ in Table A.1. in \cite{16}? (where $\theta$ denotes the local theta correspondence for the dual reductive pair $(GSO(2, 2)(F), GSp(4, F))$). The answer helps us to find out the level of $\theta((\tau_{1}, \tau_{2}))$ by using Table A.12. in \cite{16}.  \\
 
First of all we need to fix some notation. Let $\tau$ denote an irreducible admissible representation of $GL(2, F)$. Then $1 \rtimes \tau$ defines an admissible representation of $GSp(4, F)$ defined by the  induced representation from Klingen parabolic subgroup $Q$ of $GSp(4, F)$. It is known that $1 \rtimes \tau$ is reducible. It has two constituents such that 

\begin{equation}
\label{17}
 1 \rtimes \tau = \pi_{gen}(\tau) \oplus \pi_{ng}(\tau), 
\end{equation} 
where $\pi_{gen}(\tau)$ denotes the generic irreducible constituent and $\pi_{ng}(\tau)$ is the non-generic irreducible constituent.\\

Also, let $\pi$ be one of the representations $\chi_{1} \times \chi_{2} \rtimes \sigma$, $\pi \rtimes \sigma$, or $\sigma \rtimes \pi$ such that obtained by parabolic induction from $B$, $P$, or $Q$ respectively. Let assume that $\pi$ be a reducible representation. According to the structure of $\pi$, we denote the unique irreducible quotient of $\pi$ (Langlands Quotient) by $J_{B}(\chi_{1}, \chi_{2}, \sigma)$, $J_{P}(\pi, \sigma)$, and $J_{Q}(\sigma, \pi)$ for $\chi_{1} \times \chi_{2} \rtimes \sigma$, $\pi \rtimes \sigma$, and $\sigma \rtimes \pi$ respectively.\\ 

The following theorem will help us to find out the answer of the question about $\theta((\tau_{1}, \tau_{2}))$.

\begin{theorem}[W. T. Gan and S. Takeda] 
\label{th9}
Let $(\tau_{1}, \tau_{2})$ be an irreducible representation of $GSO(2, 2)(F)$ and let $\theta((\tau_{1}, \tau_{2}))$ be the local theta lift to $GSp(4, F)$. Then $\theta((\tau_{1}, \tau_{2})) = \theta((\tau_{2}, \tau_{1}))$ can be determined as follows.
\begin{itemize}
\item[(i)] If $\tau_{1} = \tau_{2} = \tau$ is a discrete series representation, then 
\begin{equation*} 
\theta((\tau, \tau)) = \pi_{gen}(\tau),
\end{equation*}
where  $\pi_{gen}(\tau)$ is the unique generic constituent of $1 \rtimes \tau$.

\item[(ii)] If $\tau_{1} \neq \tau_{2}$ are both supercuspidal, then $\theta((\tau_{1}, \tau_{2}))$ is supercuspidal.

\item[(iii)] If $\tau_{1}$ is supercuspidal and $\tau_{2} = \chi \St_{GL(2)}$, then 
\begin{equation*}
\theta((\tau_{1}, \tau_{2})) = \textbf{St}(\tau_{1}\otimes \chi^{-1}, \chi),
\end{equation*} 
where $\textbf{St}(\tau_{1}\otimes \chi^{-1}, \chi)$ is the unique irreducible constituent of $(\tau_{1}\otimes \chi^{-1})\nu^{1/2} \rtimes \chi \nu^{-1/2}$.

\item[(iv)] Suppose that $\tau_{1} = \chi_{1}\St_{GL(2)} $ and $\tau_{2} = \chi_{2}\St_{GL(2)} $, with $\chi_{1} \neq \chi_{2}$, so that $\chi_{1}^{2} = \chi_{2}^{2}$. Then
\begin{equation*}
\theta((\tau_{1}, \tau_{2})) = \textbf{St}(\chi_{1}/\chi_{2}\St_{GL(2)}, \chi_{2}) = \textbf{St}(\chi_{2}/\chi_{1}\St_{GL(2)}, \chi_{1}),
\end{equation*}
where $\textbf{St}(\chi_{1}/\chi_{2}\St_{GL(2)}, \chi_{2})$ is defined same as the previous case.

\item[(v)] Suppose that $\tau_{1}$ is a discrete representation and $\tau_{2} \hookrightarrow \pi(\chi, \chi^{'})$ with \\ $\mid \chi / \chi^{'} \mid = \mid - \mid^{-s}$ and $s \geq 0$, so that $\tau_{2}$ is non discrete series. Then
\begin{equation*}
\theta((\tau_{1}, \tau_{2})) = J_{P}(\tau_{1}\otimes \chi^{-1}, \chi).
\end{equation*}

\item[(vi)] Suppose that 
\begin{equation*}
\tau_{1} \hookrightarrow \pi(\chi_{1}, \chi_{1}^{'}) \quad \text{and} \quad \tau_{2} \hookrightarrow \pi(\chi_{2}, \chi_{2}^{'}) 
\end{equation*}
with
\begin{equation*}
\mid \chi_{i}/ \chi_{i}^{'} \mid = \mid - \mid^{-s_{i}} \quad \text{and} \quad s_{1} \geq s_{2} \geq 0.
\end{equation*}
Then
\begin{equation*}
\theta((\tau_{1}, \tau_{2})) = J_{B}(\chi_{2}^{'} / \chi_{1}, \chi_{2} / \chi_{1}, \chi_{1}).
\end{equation*}
In particular, the map $(\tau_{1}, \tau_{2}) \mapsto \theta((\tau_{1}, \tau_{2}))$ defines an injection 
\begin{equation*}
\Pi(GSO(2, 2))/ \sim \hookrightarrow \Pi(GSp(4)),
\end{equation*}
where $\sim$ denotes the relation $(\tau_{1}, \tau_{2}) \sim (\tau_{2}, \tau_{1})$.   
\end{itemize}    
\end{theorem}  

\begin{remark}
\label{r14}
W. Gan and S. Takeda consider the identification \linebreak $(GL(2) \times GL(2))/ \lbrace(z, z^{-1})\rbrace \cong GSO(2, 2)$. However, in (\ref{7}), we consider the identification  $(GL(2) \times GL(2))/ \lbrace(z, z)\rbrace \cong GSO(2, 2)$. Therefore, in our setting the local theta correspondence is injective under the relation $(\tau_{1}, \tau_{2}) \sim (\tau_{2}^{\vee}, \tau_{1}^{\vee})$,  where $\tau^{\vee}(g) = \tau((g^{-1})^{t})$.\\
\end{remark}

This result will solve the second problem stated at the beginning of this section. However, we have to precisely determine $\theta((\tau_{1}, \tau_{2}))$ according to Table A.1. in \cite{16}. In the next part we are going to complete this calculation and determine the level of $\theta((\tau_{1}, \tau_{2}))$ in the sense of paramodular representation by using Theorem \ref{th9}.

\subsection{The level of $\theta((\tau_{1}, \tau_{2}))$} 
\label{subsec3}

In this part we present the second main result in this paper. Consider the dual reductive pair $(GSO(2, 2)(F), GSp(4, F))$ over $F$. We know that $GSO(2, 2)(F)$ is isomorphic to $(GL(2, F) \times GL(2, F))/ F^{\times}$. Therefore, any irreducible admissible representation of $GSO(2, 2)(F)$ is a pair $(\tau_{1}, \tau_{2})$ where $\tau_{i}$'s are irreducible admissible representations of $GL(2, F)$ such that $\omega_{\tau_{1}} \omega_{\tau_{2}} = 1$. We have the image of local theta lift to $GSp(4, F)$ of $(\tau_{1}, \tau_{2})$ by Theorem \ref{th9}. This means that if $\theta((\tau_{1}, \tau_{2}))$ is paramodular, we can find its level by theory of local newforms for $GSp(4, F)$. Explicitly we want to answer this question ``what is the level of $\theta((\tau_{1}, \tau_{2}))$ as a paramodular representation of $GSp(4, F)$ if the level of $\tau_{1}$ and $\tau_{2}$ are $a(\tau_{1})$ and $a(\tau_{2})$ respectively?".\\
 
\noindent Fix $(\tau_{1}, \tau_{2})$ as an irreducible admissible representation of $GSO(2, 2)(F)$ such that $\omega_{\tau_{i}}$'s are trivial and $F = \mathbb{Q}_{\upsilon}$ is a non-archimedean local field. Let assume $\theta((\tau_{1}, \tau_{2})) \neq 0$.  
We are going to find the level of $\theta((\tau_{1}, \tau_{2}))$ by following cases according all possible choices for $\tau_{i}$'s.

\begin{itemize}
\item[(I)] Let $\tau_{1} = \tau_{2} = \tau$ be an irreducible admissible representation of $GL(2, F)$ of level $a(\tau)$. 

\begin{itemize}
\item[(Ia)] Let assume $\tau$ be a discrete series representation. Then $\theta((\tau, \tau)) = \pi_{gen}(\tau)$. This is the unique generic constituent of $1 \rtimes \tau$, therefore, this is the case VIIIa in Table A.12. in \cite{16}. Therefore if the level of $\tau$ is $a(\tau)$, the level of $\theta((\tau, \tau))$ is $2a(\tau)$.

\item[(I$\text{b}_{1}$)] Let $\tau \hookrightarrow \pi(\chi, \chi^{'})$ with $\mid \chi / \chi^{'} \mid = \mid - \mid^{-s}$ and $s \geq 0$. Therefore $\tau$ is a non-discrete series representation. By Theorem \ref{th9}, $\theta((\tau, \tau))= J_{B}(\chi^{'} / \chi, 1 , \chi)$. Now  $\chi^{'} / \chi \times 1 \rtimes \chi$ is irreducible if $\chi^{'} / \chi \neq \nu^{ \pm 1}$ (see \cite{16}, p. 37). This means that $J_{B}(\chi^{'} / \chi, 1 , \chi) = \chi^{'} / \chi \times 1 \rtimes \chi $. In this case the level of $\theta((\tau, \tau))$ is $a(\chi^{'}) + 3a(\chi)$, according to the Table A.12. Type I.

\item[(I$\text{b}_{2}$)]  Let $\tau \hookrightarrow \pi(\chi, \chi^{'})$ with $\mid \chi / \chi^{'} \mid = \mid - \mid^{-s}$ and $s \geq 0$, therefore, $\tau$ is non-discrete series. Therefore $\theta((\tau, \tau))= J_{B}(\chi^{'} / \chi, 1 , \chi)$ according to Theorem \ref{th9}. Now  $\chi^{'} / \chi \times 1 \rtimes \chi$ is reducible if $\chi^{'} / \chi = \nu^{ \pm 1}$ (see \cite{16}, p. 37). In this case 
\begin{equation*}
J_{B}(\chi^{'} / \chi, 1 , \chi) = J_{P}(\chi^{-1}\chi \nu^{1/2}, \chi).
\end{equation*}
This is just a simple consequence of the equation at page 21 in \cite{17}. This means that $\theta((\tau, \tau)) = J_{P}(\nu^{1/2}\triv_{GL(2)} , \chi) $. Now let $\sigma$ be the character such that $\chi = \nu^{-1/2}\sigma$. According to 2.11. in \cite{16} (see p. 39) we have 
\begin{equation*}
\theta((\tau, \tau)) = J_{P}(\nu^{1/2}\triv_{GL(2)} , \chi) = L(\nu,1_{F^\times}\rtimes\nu^{-1/2}\sigma).
\end{equation*}
This is Type VId. Therefore, the level of $\theta((\tau, \tau))$ is 0(spherical) if $\sigma$ is unramified. If $\sigma$ is ramified then $\theta((\tau, \tau))$ is not paramodular.  
\end{itemize}

\item[(II)]  If $\tau_{1} \neq \tau_{2}$ and $\tau_{i}$'s are supercuspidal, then $\theta((\tau_{1}, \tau_{2}))$ is supercuspidal. In this case the level of $\theta((\tau_{1}, \tau_{2}))$ is greater than or equal 2. The answer is not precise in this case because there is no classification for supercuspidal representations of $GSp(4, F)$.

\item[(III)] If $\tau_{1}$ is supercuspidal and $\tau_{2} = \chi\St_{GL(2)}$, then 
\begin{equation*}
\theta((\tau_{1}, \tau_{2})) = \textbf{St}(\tau_{1}\otimes \chi^{-1}, \chi),
\end{equation*} 
where $\textbf{St}(\tau_{1}\otimes \chi^{-1}, \chi)$ is the unique generic irreducible constituent of \\ $(\tau_{1}\otimes \chi^{-1})\nu^{1/2} \rtimes \chi \nu^{-1/2}$ (see \cite{16}, p. 40). Note that we assumed $\omega_{\tau_{1}} = 1$. This means that $\theta((\tau_{1}, \tau_{2}))$ is the case XIa in the Table A.12. Therefore,
\begin{equation*}
\theta((\tau_{1}, \tau_{2})) = \delta(\nu^{1/2}(\tau_{1} \otimes \chi^{-1}),\nu^{-1/2}\chi). 
\end{equation*}
If $\tau_{1}$ has level $a(\tau_{1})$, then level of $\theta((\tau_{1}, \tau_{2}))$ is\\
\begin{equation*}
\begin{cases} a(\chi (\tau_{1} \otimes \chi^{-1})) + 1 &\mbox{if } \chi \quad \text{is unramified} \\
a(\chi (\tau_{1} \otimes \chi^{-1})) + 2a(\chi) & \mbox{if } \chi \quad \text{is ramified}. \end{cases}   
\end{equation*}

\item[(IV)] Suppose that $\tau_{1} = \chi_{1}\St_{GL(2)} $ and $\tau_{2} = \chi_{2}\St_{GL(2)} $, with $\chi_{1} \neq \chi_{2}$, so that $\chi_{1}^{2} = \chi_{2}^{2}$. Then
\begin{equation*}
\theta((\tau_{1}, \tau_{2})) = \textbf{St}(\chi_{1}/\chi_{2}\St_{GL(2)}, \chi_{2}) = \textbf{St}(\chi_{2}/\chi_{1}\St_{GL(2)}, \chi_{1}).
\end{equation*}

Let $\xi = \chi_{2} / \chi_{1}$ and $\sigma = \chi_{1}$. Therefore we have 
\begin{equation*}
\theta((\tau_{1}, \tau_{2})) = \delta([\xi,\nu\xi],\nu^{-1/2}\sigma),
\end{equation*}
which is Type Va in the Table A.12. This means that the level of $\theta((\tau_{1}, \tau_{2}))$ is \\
 
\begin{equation*}
\begin{cases} 2 &\mbox{if } \sigma \quad \text{and} \quad \xi \quad \text{are unramified} \\
2a(\xi) + 1 & \mbox{if }\sigma \quad \text{is unramified and} \quad \xi \quad \text{is ramified}\\
2a(\sigma) + 1 & \mbox{if }\sigma \quad \text{is ramified and} \quad \sigma\xi \quad \text{is unramified}\\
2a(\xi \sigma) + 2a(\sigma) & \mbox{if }\sigma \quad \text{is ramified and} \quad
 \sigma\xi \quad \text{is ramified}.\\ 
\end{cases}
\end{equation*}

\item[(V)] Suppose that $\tau_{1}$ is a discrete representation and $\tau_{2} \hookrightarrow \pi(\chi, \chi^{'})$ with \\ $\mid \chi / \chi^{'} \mid = \mid - \mid^{-s}$ and $s \geq 0$, so that $\tau_{2}$ is non-discrete series. Then
\begin{equation*}
\theta((\tau_{1}, \tau_{2})) = J_{P}(\tau_{1}\otimes \chi^{-1}, \chi),
\end{equation*}
where $J_{P}(\tau_{1}\otimes \chi^{-1}, \chi)$ is the unique irreducible quotient constituent of $\tau_{1}\otimes \chi^{-1} \rtimes \chi$.
We consider different cases for $\tau_{1}$. \\
\begin{itemize}
\item[($\text{V}_{1}$)] Let $\tau_{1}$ be supercuspidal of level $a(\tau)$ and $\tau$ is not of the form $\nu^{\pm 1/2}\tau^{'}$, where $\tau^{'}$ is a supercuspidal representation with trivial central character. Therefore, $\tau_{1}\otimes \chi^{-1} \rtimes \chi$ is irreducible (see \cite{16}, p. 40) and 
\begin{equation*}
J_{P}(\tau_{1}\otimes \chi^{-1}, \chi) = \tau_{1}\otimes \chi^{-1} \rtimes \chi.  
\end{equation*}
This is Type X. Therefore, the level of $\theta((\tau_{1}, \tau_{2}))$ is $a(\chi (\tau_{1} \otimes \chi^{-1})) + 2a(\chi)$.
\item[($\text{V}_{2}$)] Let $\tau_{1}$ be supercuspidal of level $a(\tau)$ and $\tau$ is of the form $\nu^{1/2}\tau^{'}$, where $\tau^{'}$ is a supercuspidal representation with trivial central character. Let $\sigma$ be a character such that $\chi = \nu^{-1/2}\sigma$. Therefore 
\begin{equation*}
J_{P}(\tau_{1}\otimes \chi^{-1}, \chi) = L(\nu^{1/2}\tau^{'},\nu^{-1/2}\sigma).
\end{equation*}
This is Type XIb. This means that the level of $\theta((\tau_{1}, \tau_{2}))$ is $a(\sigma\tau^{'})$ if $\sigma$ is unramified. If $\sigma$ is ramified then $\theta((\tau_{1}, \tau_{2}))$ is not paramodular according to Table A.12. in \cite{16}.

\item[($\text{V}_{3}$)] Let $\tau_{1}$ be $\chi^{'}\St_{GL(2)}$. Then $\theta((\tau_{1}, \tau_{2}))$ is $J_{P}(\chi^{'}\St_{GL(2)}\otimes \chi^{-1}, \chi)$. We assume $\chi^{'}\chi^{-1} \neq \nu^{\pm1/2}$, then $\chi^{'}\St_{GL(2)}\otimes \chi^{-1} \rtimes \chi$ is irreducible. Therefore $\theta((\tau_{1}, \tau_{2})) = \chi^{'}\St_{GL(2)}\otimes \chi^{-1} \rtimes \chi $. This is Type X. Therefore, the level of $\theta((\tau_{1}, \tau_{2}))$ is $a(\chi (\tau_{1} \otimes \chi^{-1})) + 2a(\chi)$.

\item[($\text{V}_{4}$)] Let $\tau_{1}$ be $\chi^{'}\St_{GL(2)}$. Then $\theta((\tau_{1}, \tau_{2}))$ is $J_{P}(\chi^{'}\St_{GL(2)}\otimes \chi^{-1}, \chi)$. We assume $\chi^{'}\chi^{-1} = \nu^{\pm1/2}$. Therefore $\nu^{1/2}\St_{GL(2)}\rtimes \chi$ is reducible. Let $\sigma$ be a character such that $\chi = \nu^{-1/2}\sigma $, then 
\begin{equation*}
\theta((\tau_{1}, \tau_{2})) = L(\nu^{1/2}\,\St_{GL(2)},\nu^{-1/2}\sigma), 
\end{equation*}               
according to 2.11 in \cite{16} (p. 39). This is Type VIc. This means that the level of $\theta((\tau_{1}, \tau_{2}))$ is $1$ if $\sigma$ is unramified. If $\sigma$ is ramified then $\theta((\tau_{1}, \tau_{2}))$ is not paramodular. 

\item[($\text{V}_{5}$)] Let $\tau_{1}$ be $\pi(\chi_{1}, \chi_{2})$ such that $\tau_{1}$ is a discrete series. $J_{P}(\tau_{1}\otimes \chi^{-1}, \chi)$ is the unique quotient constituent of $\chi^{-1} \pi(\chi_{1}, \chi_{2}) \rtimes \chi$. The central character of $\chi^{-1} \pi(\chi_{1}, \chi_{2})$ is $\omega = \chi^{-2}\chi_{1}\chi_{2}$. Let assume $ \tau = \chi^{-1} \pi(\chi_{1}, \chi_{2})$ such that it is not of the form $\nu^{\pm 1/2}\tau^{'}$, where $\tau^{'}$ is a representation with trivial central character. Then $\chi^{-1} \pi(\chi_{1}, \chi_{2}) \rtimes \chi$ is irreducible and $\theta((\tau_{1}, \tau_{2})) = \chi^{-1} \pi(\chi_{1}, \chi_{2}) \rtimes \chi$. Therefore, the level of $\theta((\tau_{1}, \tau_{2})$ is $a(\tau) + 2a(\chi)$.

\item[($\text{V}_{6}$)] Let $\tau_{1}$ be $\pi(\chi_{1}, \chi_{2})$ such that $\tau_{1}$ is a discrete series. $J_{P}(\tau_{1}\otimes \chi^{-1}, \chi)$ is the unique quotient constituent of $\chi^{-1} \pi(\chi_{1}, \chi_{2}) \rtimes \chi$. The central character of $\chi^{-1} \pi(\chi_{1}, \chi_{2})$ is $\omega = \chi^{-2}\chi_{1}\chi_{2}$. Let assume $ \tau = \chi^{-1} \pi(\chi_{1}, \chi_{2})$ such that it is of the form $\nu^{\pm 1/2}\tau^{'}$, where $\tau^{'}$ is a representation with trivial central character. Then $\nu^{1/2}\tau^{'} \rtimes \chi$ is reducible and
\begin{equation*}
\theta((\tau_{1}, \tau_{2})) = L(\nu^{1/2}\tau^{'},\nu^{-1/2}\sigma).
\end{equation*}
  This means that the level of $\theta((\tau_{1}, \tau_{2}))$ is $a(\sigma\tau^{'})$ if $\sigma$ is unramified. If $\sigma$ is ramified then $\theta((\tau_{1}, \tau_{2}))$ is not paramodular according to Table A.12. in \cite{16}    
\end{itemize} 

\item[(VI)]Suppose that 
\begin{equation*}
\tau_{1} \hookrightarrow \pi(\chi_{1}, \chi_{1}^{'}) \quad \text{and} \quad \tau_{2} \hookrightarrow \pi(\chi_{2}, \chi_{2}^{'}) 
\end{equation*}
with
\begin{equation*}
\mid \chi_{i}/ \chi_{i}^{'} \mid = \mid - \mid^{-s_{i}} \quad \text{and} \quad s_{1} \geq s_{2} \geq 0.
\end{equation*}
Then
\begin{equation*}
\theta((\tau_{1}, \tau_{2})) = J_{B}(\chi_{2}^{'} / \chi_{1}, \chi_{2} / \chi_{1}, \chi_{1}).
\end{equation*}

We need to consider several cases according to reducibility of $\chi_{2}^{'} / \chi_{1} \times \chi_{2} / \chi_{1} \rtimes \chi_{1}$.

\begin{itemize}
\item[($\text{VI}_{1}$)] If $\chi_{2}^{'} / \chi_{1} \neq \nu^{\pm 1}$, $\chi_{2} / \chi_{1} \neq \nu^{\pm 1}$ and $\chi_{2}^{'} / \chi_{1} \neq \nu^{\pm 1} \chi_{2} / \chi_{1} $, then $\chi_{2}^{'} / \chi_{1} \times \chi_{2} / \chi_{1} \rtimes \chi_{1}$ is irreducible, otherwise, it is reducible. If $\chi_{2}^{'} / \chi_{1} \times \chi_{2} / \chi_{1} \rtimes \chi_{1}$ is irreducible, then 
\begin{equation*}
\theta((\tau_{1}, \tau_{2})) = \chi_{2}^{'} / \chi_{1} \times \chi_{2} / \chi_{1} \rtimes \chi_{1}.
\end{equation*}
This is Type I in the Table A.12. in \cite{16}. Therefore the level of $\theta((\tau_{1}, \tau_{2}))$ is $a(\chi_{2}^{'}) + a(\chi_{2}) + 2a(\chi_{1})$.\\

We are considering cases that $\chi_{2}^{'} / \chi_{1} \times \chi_{2} / \chi_{1} \rtimes \chi_{1}$ is reducible.       

\item[($\text{VI}_{2}$)] Let $\chi$ be character of $F^{\times}$ such that $\chi \neq \nu^{\pm 3/2}$ and $\chi^{2} \neq \nu^{\pm 1}$. Let assume $\chi_{2}^{'} / \chi_{1} = \nu^{1/2}\chi$ and $\chi_{2} / \chi_{1} = \nu^{-1/2}\chi$. Therefore $\chi_{2}^{'} / \chi_{1} \times \chi_{2} / \chi_{1} \rtimes \chi_{1}$ looks like $\nu^{1/2}\chi \times  \nu^{-1/2}\chi \rtimes \chi_{1}$. Using Lemma 3.3. in \cite{20}, we have 
\begin{equation*}
\theta((\tau_{1}, \tau_{2})) = \chi \triv_{GL(2)} \rtimes \chi_{1}.
\end{equation*} 
This is Type IIb in the Table A.12. in \cite{16}. Therefore the level of $\theta((\tau_{1}, \tau_{2}))$ is 
$2a(\chi_{1})$ if $\chi_{1}\chi$ is unramified. If $\chi_{1}\chi$ is ramified then  $\theta((\tau_{1}, \tau_{2}))$ is not paramodular. 

\item[($\text{VI}_{3}$)] Let $\chi$ be character of $F^{\times}$ such that $\chi \neq 1_{F^{\times}}$ and $\chi^{2} \neq \nu^{\pm 2}$. Let assume $\chi_{2}^{'} / \chi_{1} = \chi$ and $\chi_{2} / \chi_{1} = \nu$. Moreover $\sigma$ is a character such that $\chi_{1} = \nu^{-1/2}\sigma$. Therefore $\chi_{2}^{'} / \chi_{1} \times \chi_{2} / \chi_{1} \rtimes \chi_{1}$ looks like $\chi \times  \nu \rtimes \nu^{-1/2}\sigma$. Using Lemma 3.4. in \cite{20}, we have 
\begin{equation*}
\theta((\tau_{1}, \tau_{2})) = \chi \rtimes \sigma \triv_{GSp(2)}.
\end{equation*} 
This is Type IIIb in the Table A.12. in \cite{16}. Therefore the level of $\theta((\tau_{1}, \tau_{2}))$ is 0(spherical) if $\sigma$ is unramified. If $\sigma$ is ramified then  $\theta((\tau_{1}, \tau_{2}))$ is not paramodular.

\item[($\text{VI}_{4}$)] Let assume $\chi_{2}^{'} / \chi_{1} = \nu^{2}$ and $\chi_{2} / \chi_{1} = \nu$. Moreover $\sigma$ is a character such that $\chi_{1} = \nu^{-3/2}\sigma$. Therefore $\chi_{2}^{'} / \chi_{1} \times \chi_{2} / \chi_{1} \rtimes \chi_{1}$ looks like $\nu^{2} \times  \nu \rtimes \nu^{-3/2}\sigma$. Moreover, it is clear that $\mid \chi_{2} / \chi_{2}^{'} \mid = \nu^{-1}$. We have 
\begin{equation*}
J_{B}(\nu^{2}, \nu , \nu^{-3/2}\sigma) = J_{P}(\nu^{3/2}\triv_{GL(2)}, \nu^{-3/2}\sigma
).
\end{equation*}
This is a simple consequence of the equation at page 21 in \cite{17}. According to 2.9. in \cite{16} (see p. 38)
\begin{equation*}
\theta((\tau_{1}, \tau_{2})) = \sigma\triv_{GSp(4)}.
\end{equation*} 
This is Type IVd in the Table A.12. in \cite{16}. Therefore the level of $\theta((\tau_{1}, \tau_{2}))$ is 0 if $\sigma$ is unramified. If $\sigma$ is ramified then $\theta((\tau_{1}, \tau_{2}))$ is not paramodular.

\item[($\text{VI}_{5}$)]  Let assume $\chi_{2}^{'} / \chi_{1} = \nu \xi$ and $\chi_{2} / \chi_{1} = \xi$, where $\xi$ is non-trivial quadratic character of $F^{\times}$.  Moreover $\sigma$ is a character such that $\chi_{1} = \nu^{-1/2}\sigma$. Therefore, $\chi_{2}^{'} / \chi_{1} \times \chi_{2} / \chi_{1} \rtimes \chi_{1}$ looks like $ \nu \xi \times  \xi \rtimes \nu^{-1/2}\sigma$. It is clear that $\mid \chi_{2} / \chi_{2}^{'} \mid = \nu^{-1}$. We have 
\begin{equation*}
J_{B}(\nu^{2}, \nu , \nu^{-3/2}\sigma) = J_{P}(\nu^{1/2} \xi \triv_{GL(2)}, \nu^{-1/2}\sigma
).
\end{equation*}
This is a simple consequence of the equation at page 21 in \cite{17}. According to 2.10. in \cite{16}
\begin{equation*}
\theta((\tau_{1}, \tau_{2})) = L(\nu\xi,\xi\rtimes\nu^{-1/2}\sigma).
\end{equation*} 
This is Type Vd in the Table A.12. in \cite{16}. Therefore the level of $\theta((\tau_{1}, \tau_{2}))$ is 0(spherical) if $\sigma$ and $\xi$ are unramified. If $\sigma$ or $\xi$ are ramified then the $\theta((\tau_{1}, \tau_{2}))$ is not paramodular.

\item[($\text{VI}_{6}$)]  Let assume $\chi_{2}^{'} / \chi_{1} = \nu $ and $\chi_{2} / \chi_{1} = 1_{F^{\times}}$. Moreover $\sigma$ is a character such that $\chi_{1} = \nu^{-1/2}\sigma$. Therefore $\chi_{2}^{'} / \chi_{1} \times \chi_{2} / \chi_{1} \rtimes \chi_{1}$ looks like $ \nu \times 1_{F^{\times}} \rtimes \nu^{-1/2}\sigma$. It is clear that $\mid \chi_{2} / \chi_{2}^{'} \mid = \nu^{-1}$. We have 
\begin{equation*}
J_{B}(\nu^{2}, \nu , \nu^{-3/2}\sigma) = J_{P}(\nu^{1/2} \triv_{GL(2)}, \nu^{-1/2}\sigma
).
\end{equation*}
This is a simple consequence of the equation at page 21 in \cite{17}. According to 2.11. in \cite{16}
\begin{equation*}
\theta((\tau_{1}, \tau_{2})) =  L(\nu,1_{F^\times}\rtimes\nu^{-1/2}\sigma).
\end{equation*} 
This is Type VId in the Table A.12. in \cite{16}. Therefore the level of $\theta((\tau_{1}, \tau_{2}))$ is 0(spherical) if $\sigma$ is unramified. If $\sigma$ is ramified then the $\theta((\tau_{1}, \tau_{2}))$ is not paramodular.

\end{itemize} 
  
\end{itemize}

We are in the position to state the second main result in this paper. According to above calculation, we have
\begin{theorem}
\label{th10}  
 Let $F = \mathbb{Q}_{\upsilon}$ be a non-archimedean local field and we denote the normalized absolute value by $\nu$ such that $\nu(\varpi) = \mid \varpi \mid = q^{-1}$. Let consider the dual reductive pair \linebreak $(GSO(2, 2)(F), GSp(4, F))$ over $F$. Let $(\tau_{1}, \tau_{2})$ be an irreducible non-supercuspidal admissible representation of $GSO(2, 2)(F)$ such that $\tau_{i}$ has level $a(\tau_{i})$ as an irreducible admissible representation of $GL(2, F)$ such that they have trivial central character. If $\theta((\tau_{1}, \tau_{2}))$ denotes the local theta lift to $GSp(4, F)$, then we can calculate the level of $\theta((\tau_{1}, \tau_{2}))$ as an admissible representation of $GSp(4, F)$, according to $a(\tau_{i})$'s.      
\end{theorem}
\begin{proof}
 If $(\tau_{1}, \tau_{2})$ is an irreducible representation of $GSO(2, 2)(F)$, we are able to find the level of $\theta((\tau_{1}, \tau_{2}))$, according to one of the above cases. Note that we need to assume the non-supercuspidal representations of $GSO(2, 2)(F)$ because the case II does not provide explicit information.
 \renewcommand{\qedsymbol}{} 
\end{proof}

 \begin{remark}
 \label{r15}
 In the above list in some cases the local theta lift were not paramodular. Note that because we are interested in studying global representation $\pi = \pi(\tau_{1}, \tau_{2})$ of $GSO(2, 2)(\mathbb{A})$ which is generic, then the finite constituent $\Pi_{\sigma^{+}, \upsilon}$   is generic for all finite places $\upsilon$. Therefore, we never face those cases that the local theta lift to $GSp(4, F)$ are not paramodular.
 \end{remark}
  
Let $\pi = \pi(\tau_{1}, \tau_{2})$ be a cuspidal representation of $(GL(2, \mathbb{A}) \times GL(2, \mathbb{A}))/ \mathbb{A}^{\times}$ of level $(N_{1}, N_{2})$. If we assume that $\pi = \otimes_{\upsilon} \pi_{\upsilon}$, then we know that $\Pi_{\sigma^{+}} = \otimes_{\upsilon} \theta(\pi_{\upsilon}^{+})^{*}$. This means that $\Pi_{\sigma^{+}}$ is determined completely by the local theta lifting of $\pi_{\upsilon}^{+}$'s. Therefore, because we are able to figure out the level of all finite constituent $\Pi_{\sigma^{+}, \upsilon}$ for each non-archimedean local place, then we are able to calculate the level of $\Pi_{\sigma^{+}}$ in terms of the level of finite part $\Pi_{\sigma^{+}, fin}$ of $\Pi_{\sigma^{+}}$. This means that we can have the following corollary.

\begin{corollary}
\label{c3}
Let $\pi = \pi(\tau_{1}, \tau_{2})$ be a cuspidal automorphic representation of $(GL(2, \mathbb{A}) \times GL(2, \mathbb{A}))/ \mathbb{A}^{\times}$ of level $(N_{1}, N_{2})$. Then we are able to find the level of $\Pi_{\sigma^{+}, fin}$ as the finite part of the paramodular representation $\Pi_{\sigma^{+}}$ of $GSp(4, \mathbb{A})$. 
\end{corollary}

\subsection{Some results on the strict endoscopic part}
The above cases give us an explicit way to calculate the level of local theta lift from $GSO(2, 2)(F)$ to $GSp(4, F)$ although the calculation is massive and complicated. In the next section we present some useful examples for Theorem \ref{th10}. Moreover, Section \ref{sec9} contains an application of this method to prove the conjecture on strict endoscopic part in the case of trivial level structure. Now we close this section by relating the results in this section with the strict endoscopic part of $H_{!}^{3}(X_{N}, \mathbb{V}_{l, m})$ for a fix local system $\mathbb{V}_{l, m}$.\\

\noindent For any non-negative integer $N$, consider the compact open subgroup $K_{N}$ defined by $K_{N} = \prod_{p \mid N} K(\mathcal{P}^{ord_{N}(p)})$, where $\mathcal{P}$ is generated by prime number $p$ and $K(\mathcal{P}^{n})$ is defined in (\ref{16}). Let $X_{K_{N}} := _{K_{N}}M_{G}$ denote the complex points of the $K_{N}$-component of the Shimura variety $M_{G}$.
Under the isomorphism in (\ref{5}), the strict endoscopic part of $H_{!}^{3}(X_{K_{N}}, \mathbb{V}_{l, m})$ which is denoted by $H^{3}_{\text{End}^{\text{s}}}(X_{K_{N}}, \mathbb{V}_{l, m})$ is defined by 
\begin{equation}
\label{18}
H^{3}_{\text{End}^{\text{s}}}(X_{K_{N}}, \mathbb{V}_{l, m}) \cong \bigoplus_{\pi = \pi_{\infty}\otimes \pi_{fin}} m(\pi)(\pi_{fin})^{K_{N}}, 
\end{equation}
 where the sum runs over all irreducible cuspidal automorphic representations of $G(\mathbb{A})$ such that $\pi_{\infty}$ is a discrete series representation of $GSp(4, \mathbb{R})^{+}$ isomorphic with $\pi_{\Lambda + \rho}^{(1, 2 )}$ or $\pi_{\Lambda + \rho}^{(2, 1)}$. Here $(\pi_{fin})^{K_{N}}$ denotes the invariant elements of $V_{fin} = \bigotimes_{\text{$\upsilon$ is finite place}} V_{\upsilon}$ under the compact open subgroup $K_{N}$.\\
 
\noindent Now we define the set 
\begin{equation}
\label{19}
\mathbb{V}_{\Theta}(K_{N}) := \lbrace \Theta(f_{1}, f_{2}) = \Pi_{\sigma^{+}, fin} \mid \Pi_{\sigma^{+}, fin} \quad \text{has level $K_{N}$}  \rbrace,
\end{equation} 
where $\Theta$ is defined in (\ref{15}) and $(f_{1}, f_{2})$ is a pair of newforms of $GL(2)$ with trivial central characters such that $f_{1}$ has the weight $l + m + 4$ and $f_{2}$ has the weights $l -m +2$ or $m - l -2$. Therefore $\dim((\Pi_{\sigma^{+}, fin})^{K_{N}}) = 1$. Let 
\begin{equation}
\label{20}
V_{\Theta}(K_{N}) = \bigoplus_{\text{Over the set $\mathbb{V}_{\Theta}(K_{N})$}} m(\pi)(\Pi_{\sigma^{+}, fin})^{K_{N}}.
\end{equation}
Now it is well-known that $H_{!}^{3}(X_{K_{N}}, \mathbb{V}_{l, m})$ is finite dimensional. This implies that $H^{3}_{\text{End}^{\text{s}}}(X_{K_{N}}, \mathbb{V}_{l, m})$ is finite dimensional too. Also we know that $V_{\Theta}(K_{N})$ is a non-zero subspace of the strict endoscopic part $H^{3}_{\text{End}^{\text{s}}}(X_{K_{N}}, \mathbb{V}_{l, m})$ (Theorem \ref{th5}). Moreover if we fix $K_{N}$, then by checking the cases in the previous part, we are able to find all possible choices of $(f_{1}, f_{2})$ that $\Theta(f_{1}, f_{2})$ is an element of $\mathbb{V}_{\Theta}(K_{N})$. Therefore we have 

\begin{corollary}
\label{c4}
Let $\mathbb{V}_{l, m}$ be a local system on $X_{K_{N}}$ such that $l - m \geq 2$ and $m \geq 2$. Then ,with the above notation, for the compact open level subgroup  $K_{N}$ of $GSp(\mathbb{A})$, we have $V_{\Theta}(K_{N})$ is a non-zero subspace of $H^{3}_{\text{End}^{\text{s}}}(X_{K_{N}}, \mathbb{V}_{l, m})$. Also dimension of $V_{\Theta}(K_{N})$ over $\mathbb{C}$ is determined completely for fixed $K_{N}$.
\end{corollary}
\begin{remark}
\label{r16}
Note that we did not provide a formula for dimension of  $V_{\Theta}(K_{N})$. Practically, for large $N$ it is very complicated to find this dimension because of massive calculation. However, we will see in Section \ref{sec9} that for small levels we can prove non-trivial results by using this approach.
\end{remark} 
\begin{remark}
\label{r17}
Note that $X_{N} := \mathcal{A}_{2, N}(\mathbb{C})$ is embedded into  $K(N)$-component of the Shimura variety $M_{G}$. Moreover $K(N)$ is a proper finite index subgroup of $K_{N}$ where $K(N)$ defined in (\ref{level}). Under the isomorphism in (\ref{5}) it is clear that if $(\pi_{fin})^{K_{N}}$ contributes in $H^{3}_{\text{End}^{\text{s}}}(X_{K_{N}}, \mathbb{V}_{l, m})$  under the isomorphism in (\ref{5}), then it will contribute in $H^{3}_{\text{End}^{\text{s}}}(X_{N}, \mathbb{V}_{l, m})$. This means that $V_{\Theta}(K_{N})$ is a subspace of $H^{3}_{\text{End}^{\text{s}}}(X_{N}, \mathbb{V}_{l, m})$. Therefore we can find a non-zero subspace with determined dimension of $H^{3}_{\text{End}^{\text{s}}}(X_{N}, \mathbb{V}_{l, m})$. Note that if we fix $K(N)$ (as we did with $K_{N}$), then we can define $\mathbb{V}_{\Theta}(K(N))$ as we did in (\ref{19}). But $\Theta(f_{1}, f_{2})$'s can have different paramodular levels. Moreover $(\Pi_{\sigma^{+}, fin})^{K(N)}$ is not one dimensional. This shows that we face extra difficulties to calculate the dimension of $V_{\Theta}(K(N))$. Finally we should mention that for trivial level structure $N = 0$, $K(N) = K_{N} = GSp(4, \widehat{\mathbb{Z}})$. We will use this fact in Section \ref{sec9}.     
\end{remark}

\section{Examples}
\label{sec8}
This section contains some examples in which we apply the method in Section \ref{sec7} to practice the calculation of the level structure of theta lift of a pair of representations of $GL(2, \mathbb{A})$, $(\tau_{1}, \tau_{2})$, to $GSp(4, \mathbb{A})$. We hope that it clarifies in greater details the machinery that explained in Section \ref{sec7}. 

\subsection{Unramified Representations}
We first calculate the level of theta lift of $\pi = \pi(\tau_{1}, \tau_{2})$, where $(\tau_{1}, \tau_{2})$ denotes a pair of cuspidal representations of $GL(2, \mathbb{A})$ with trivial central characters corresponded to a pair of newforms $(f_{1}, f_{2})$ of weights $k$ and $l$ ($k \neq l$) with trivial level.\\

\noindent Let $\tau_{i} = \otimes_{\upsilon} \tau_{i, \upsilon}$ ($i = 1, 2$). Because we assume that $f_{i}$'s have trivial levels, therefore each $\tau_{i, \upsilon}$ is an unramified representation of $GL(2, \mathcal{P}_{\upsilon})$. By the classical Jacquet-Langlands theory (\cite{15}), we know that every irreducible admissible unramified representation of $GL(2, \mathcal{P}_{\upsilon})$ is isomorphic with a principal series $\pi(\chi, \chi^{'})$, where $\chi$ and $\chi^{'}$ are unramified characters of $\mathcal{P}_{\upsilon}^{\times}$ such that $\chi (\chi^{'})^{-1} \neq \nu^{\pm 1}$ . Therefore, we assume $\tau_{i, \upsilon} = \pi(\chi_{i, \upsilon}, \chi_{i, \upsilon}^{'})$ for each place $\upsilon$.\\

According to Section \ref{sec7} to find the level of the irreducible constituent, $\Pi(\pi)$, of  $\Theta(\pi)$, we need to determine the level of all local theta lifts $\theta((\tau_{1, \upsilon}, \tau_{2, \upsilon}))$ at each finite place $\upsilon$. Under the unramified assumption, it means that we should calculate the level of $\theta((\tau_{1, \upsilon}, \tau_{2, \upsilon}))$, where  $\tau_{i, \upsilon} = \pi(\chi_{i, \upsilon}, \chi_{i, \upsilon}^{'})$ such that $\chi_{i, \upsilon} (\chi^{'}_{i, \upsilon})^{-1} \neq \nu^{\pm 1}$.\\

\noindent Fix the finite place $\nu$. According to  $\tau_{i, \upsilon} = \pi(\chi_{i, \upsilon}, \chi_{i, \upsilon}^{'})$, we should consider one of the cases VI$_{1}$ to VI$_{6}$. But we assume that  $\chi_{2, \upsilon} (\chi^{'}_{2, \upsilon})^{-1} \neq \nu^{\pm 1}$, this means that $\mid \chi_{2, \upsilon}/ \chi^{'}_{2, \upsilon} \mid \neq \nu^{\pm 1}$. Therefore only the cases VI$_{1}$ and VI$_{3}$ should consider here.\\

\noindent In the case VI$_{1}$, the level of $\theta((\tau_{1, \upsilon}, \tau_{2, \upsilon}))$ is $a(\chi_{2, \upsilon}^{'}) + a(\chi_{2, \upsilon}) + 2a(\chi_{1, \upsilon})$. But all the characters have zero level (unramified). This means that $\theta((\tau_{1, \upsilon}, \tau_{2, \upsilon}))$ has trivial level, i.e., it is unramified.\\

\noindent In the case VI$_{3}$, $\theta((\tau_{1, \upsilon}, \tau_{2, \upsilon}))$ is unramified except if $\chi_{1, \upsilon} = \nu^{\pm 1}\sigma$ where $\sigma$ is ramified. But this case never happens because $\chi_{1, \upsilon}$ is unramified. Therefore, if $(\tau_{1, \upsilon}, \tau_{2, \upsilon})$ is unramified, then $\theta((\tau_{1, \upsilon}, \tau_{2, \upsilon}))$ is unramified too. In summary we have 

\begin{corollary}
\label{c5}
Let $\pi = \pi(\tau_{1}, \tau_{2})$, where $(\tau_{1}, \tau_{2})$ denotes a pair of cuspidal automorphic representations of $GL(2, \mathbb{A})$ with trivial central characters corresponded to a pair of newforms $(f_{1}, f_{2})$ of weights $k$ and $l$ ($k \neq l$) with trivial level. Then, $\Pi(\pi)$ is an unramified irreducible cuspidal automorphic representation of $GSp(4, \mathbb{A})$.\\
\end{corollary}

\subsection{The square free level}
We keep our notation in the previous part. This time we assume that $f_{1}$ has level $N_{1} = \prod_{i}^{n_{1}} p_{i}$, where $p_{i}$'s are distinct primes and $f_{2}$ has trivial level. The fact that  $f_{1}$ has the level $N_{1} = \prod_{i = 1}^{n_{1}} p_{i}$ implies that if $\upsilon \in \lbrace p_{1}, \cdots p_{n_{1}}\rbrace$, then $\tau_{1, \upsilon} = \chi_{\upsilon}\St_{GL(2)}$, where $\chi_{\upsilon}$ is an unramified character. We want to find the level of $\Pi(\pi)$ under the above assumption. If $\upsilon \not \in \lbrace p_{1}, \cdots p_{n_{1}}\rbrace $, then $(\tau_{1, \upsilon}, \tau_{2, \upsilon})$ is unramified. Therefore, $\theta((\tau_{1, \upsilon}, \tau_{2, \upsilon}))$ is unramified. Let assume that $\upsilon = p_{i}$. We should find the level of $\theta((\tau_{1, p_{i}}, \tau_{2, p_{i}}))$, where
$\tau_{1, p_{i}} =  \chi_{p_{i}}\St_{GL(2)}$ such that $\chi_{p_{i}}$ is an unramified character, and $\tau_{2, p_{i}} = \pi(\chi_{2, p_{i}}, \chi_{2, p_{i}}^{'})$. This is clear that we should follow cases V$_{3}$ and V$_{4}$ in Section \ref{sec7}.\\

\noindent If we assume $\chi_{p_{i}}\chi_{2, p_{i}}^{-1} \neq \nu^{\pm 1}$, we are in the case V$_{3}$. This means that the level of $\theta((\tau_{1, p_{i}}, \tau_{2, p_{i}}))$ is $a(\chi_{2, p_{i}}(\chi_{p_{i}}\St_{GL(2)} \otimes \chi_{2, p_{i}})) + 2a(\chi_{2, p_{i}})$. But all the characters are unramified, so the level of $\theta((\tau_{1, p_{i}}, \tau_{2, p_{i}}))$ is one at prime $p_{i}$ (see Table \ref{table3}). But if we assume $\chi_{p_{i}}\chi_{2, p_{i}}^{-1} = \nu^{\pm 1}$, the level of $\theta((\tau_{1, p_{i}}, \tau_{2, p_{i}}))$ is $a(\sigma\St_{GL(2)})$, where $\sigma$ is a character such that $\chi_{2, p_{i}} = \nu^{-1/2}\sigma$. Because $\chi_{2, p_{i}}$ is unramified, therefore, $\sigma$ is unramified. This means that the level is again one. In summary we have 

\begin{corollary}
\label{c6}
Let $\pi = \pi(\tau_{1}, \tau_{2})$, where $(\tau_{1}, \tau_{2})$ denotes a pair of cuspidal automorphic representations of $GL(2, \mathbb{A})$ with trivial central characters corresponded to a pair of newforms $(f_{1}, f_{2})$ of weights $k$ and $l$ ($k \neq l$), $f_{1}$ has level $\prod_{i = 1}^{n_{1}} p_{i}$ (distinct primes), and $f_{2}$ has trivial level. Then $\Pi(\pi)$ is an irreducible cuspidal automorphic representation of $GSp(4, \mathbb{A})$ of level $\prod_{i = 1}^{n_{1}} p_{i}$.\\
\end{corollary}

Now we consider that $f_{1}$ has the level $\prod_{i=1}^{n_{1}} p_{i}$ and $f_{2}$ has the level $\prod_{i=1}^{n_{2}} q_{i}$ such that $\lbrace p_{1}, \cdots p_{n_{1}}, q_{1}, \cdots q_{n_{2}} \rbrace$ is a set of distinct primes. In this case, it is obvious that with a same argument as in Corollary \ref{c6} we have

\begin{corollary}
\label{c7}
Let $\pi = \pi(\tau_{1}, \tau_{2})$, where $(\tau_{1}, \tau_{2})$ denotes a pair of cuspidal automorphic representations of $GL(2, \mathbb{A})$ with trivial central characters corresponded to a pair of newforms $(f_{1}, f_{2})$ of weights $k$ and $l$ ($k \neq l$), $f_{1}$ has level $N_{1} = \prod_{i = 1}^{n_{1}} p_{i}$ (distinct primes), and $f_{2}$ has level $N_{2} = \prod_{i = 1}^{n_{2}} q_{i}$ such that all primes are distinct. Then $\Pi(\pi)$ is an irreducible cuspidal automorphic representation of $GSp(4, \mathbb{A})$ of level $N_{1}N_{2}$.\\
\end{corollary}

The next interesting case is considering $\pi = \pi(\tau_{1}, \tau_{2})$, where $(\tau_{1}, \tau_{2})$ denotes a pair of cuspidal automorphic representations of $GL(2, \mathbb{A})$ with trivial central characters corresponded to a pair of newforms $(f_{1}, f_{2})$ of weights $k$ and $l$ ($k \neq l$), $f_{1}$ has level $N_{1} = \prod_{i = 1}^{n_{1}} p_{i}$ (distinct primes), and $f_{2}$ has level $N_{2} = \prod_{i = 1}^{n_{2}} q_{i}$ such that $p_{k} = q_{k}$ for some $k$'s. We are going to assume that $p_{k} = q_{k}$ at only one place. Then, we can find the general case by using induction. Fix the index $k$ such that $p_{k} = q_{k}$ only at this index. Again, if we want to find the level of $\Pi(\pi)$, we need to find the answer at each place $\upsilon$. According to Corollaries \ref{c6} and \ref{c7}, the answer is clear for all places except $\upsilon = p_{k}$. Because $\tau_{i, p_{k}}$'s have level one at $p_{k}$, it means that $\tau_{i, p_{k}} = \chi_{i, p_{k}}\St_{GL(2)}$. Therefore, it is clear that at $p_{k}$, we should consider the case IV. Note that we know that $\chi_{i, p_{k}}$ are unramified, therefore, $\sigma$ and $\xi$ in the case IV are unramified too. Therefore, the level of $\theta((\tau_{1, p_{k}}, \tau_{2, p_{k}}))$ is $2$. In summary we have

\begin{corollary}
\label{c8}
Let $\pi = \pi(\tau_{1}, \tau_{2})$, where $(\tau_{1}, \tau_{2})$ denotes a pair of cuspidal automorphic representations of $GL(2, \mathbb{A})$ with trivial central characters corresponded to a pair of newforms $(f_{1}, f_{2})$ of weights $k$ and $l$ ($k \neq l$), $f_{1}$ has level $N_{1} = \prod_{i = 1}^{n_{1}} p_{i}$ (distinct primes), and $f_{2}$ has level $N_{2} = \prod_{i = 1}^{n_{2}} q_{i}$ such that for any $j \in J$, then $p_{j} = q_{j}$. Then, $\Pi(\pi)$ is an irreducible cuspidal automorphic representation of $GSp(4, \mathbb{A})$ of level $\prod_{i \not \in J } p_{i} \prod_{i \not \in J}q_{i} \prod_{j \in J}p_{j}^{2}$.\\
\end{corollary}

\section{ Proof of Conjecture \ref{conj2}}
\label{sec9}

This section contains the proof of Conjecture \ref{conj2}. Conjecture \ref{conj2} gives an explicit motivic description of the strict endoscopic part of a local system $\mathbb{V}(\lambda)$ over $\mathcal{A}_{2} := \mathcal{A}_{2, 1}$ (i.e., the moduli space of principally polarized abelian surfaces with trivial level structure.) Specifically, in our work, we are dealing with Betti cohomology and (real) Hodge structures. Explicitly, these cohomology spaces satisfy the motivic description at (\ref{2000}).\\ 

\noindent Note that G. Faltings proved (\cite{1})
\begin{equation}
F^{3} = H_{!}^{3, 0}(X, \mathbb{V}(\lambda)) \cong S_{l - m, m + 3},
\end{equation}
where $S_{l - m, m + 3}$ defined in (\ref{Faltings}). We will prove Conjecture \ref{conj2} here. Therefore, we have
\begin{equation}
\label{30}
e_{c}(\mathcal{A}_{2}, \mathbb{V}(\lambda)) = -S[l - m, m + 3] - s_{l + m + 4}S[l - m + 2]\mathbb{L}^{m + 1} + e_{\text{Eis}}(\mathcal{A}_{2}, \mathbb{V}(\lambda)).
\end{equation}
This gives a complete structural description of compactly supported cohomology of the local system $\mathbb{V}(\lambda)$ on $\mathcal{A}_{2}$ when $\lambda$ is sufficiently regular.

\begin{theorem}
\label{MT} 
If $\mathbb{V}(\lambda)$ ($\lambda = (l, m)$) denotes a fixed local system on $\mathcal{A}_{2}$, and $\lambda$ is sufficiently regular (i.e., $l- m \geq 2$ and $m \geq 2$), then we have  
\begin{equation}
\label{2000}
e_{\text{endo}}(\mathcal{A}_{2}, \mathbb{V}(\lambda)) = -s_{l + m + 4}S[l - m + 2]\mathbb{L}^{m + 1}.
\end{equation}
\end{theorem}

\begin{proof} 
Since we know that the Euler characteristic of inner cohomology (\cite{43}) and have Tsushima's dimension formula (\cite{26} and \cite{27})\footnote{These formulas are massive and far from the scope of this  paper. Here we only restrict ourself to state this fact that it suffices to construct a subspace of dimension $2s_{l + m + 4}s_{l - m + 2}$ in the strict endoscopic part by using the map in (\ref{15}), where $s_{n} = \dim(S_{n}(\Gamma_{1}(1)))$ (see \cite{3}, p. 232).}, it suffices to construct a subspace of dimension $2s_{l + m + 4}s_{l - m + 2}$ in the strict endoscopic part by using the map in (\ref{15}), where $s_{n} = \dim(S_{n}(\Gamma_{1}(1)))$ (see \cite{3}, p. 232). 

Consider the spaces $S_{l + m + 4}(\Gamma_{1}(1))$ and $S_{l - m + 2}(\Gamma_{1}(1))$. These are spaces of cusp forms with trivial central characters and trivial levels. According to classical theory (\cite{15}), we can find bases $\lbrace f_{1}, \cdots f_{s_{l + m + 4}}\rbrace$
and $\lbrace g_{1}, \cdots g_{s_{l - m + 2}}\rbrace$ of newforms of $GL(2)$ for the spaces $S_{l + m + 4}(\Gamma_{1}(1))$ and $S_{l - m + 2}(\Gamma_{1}(1))$ respectively. Consider a pair $(f_{i}, g_{j})$ for $1 \leq i \leq l + m + 4$ and $1 \leq j \leq l - m + 2$. Fix the pair $(f_{i}, g_{j})$. We assign to $(f_{i}, g_{j})$ the cuspidal automorphic representation  $\pi = \pi(\tau_{f_{i}}, \tau_{g_{j}})$ of $GSO(2, 2)(\mathbb{A})$. Let $\sigma_{f_{i}, g_{j}}^{+}$ denote the extension of $\pi$ to $GO(2, 2)$. Note that $(\sigma_{f_{i}, g_{j}})_{\infty}^{+} \cong \pi_{l + m + 4, l - m + 2}$ and $\sigma_{f_{i}, g_{j}} \in \mathcal{A}_{0}(H)$, where $H = GO(2, 2)$. Now $\Theta((\sigma^{+}_{f_{i}, g_{j}}))$, defined in (\ref{15}), gives a non-zero element in the strict endoscopic part. Moreover, Corollary \ref{c5} shows that $\Theta((\sigma^{+}_{f_{i}, g_{j}}))$ is invariant under level subgroup $K_{0}$ because $\pi$ is an unramified representation. Therefore, $\Theta((\sigma^{+}_{f_{i}, g_{j}}))$ contributes in the strict endoscopic part of $H^{3}_{!}(\mathcal{A}_{2}, \mathbb{V}(\lambda))$. Moreover, according to Theorem \ref{th9} and Corollary \ref{c2}, the map $(f_{i}, g_{j}) \to \Theta((\sigma^{+}_{f_{i}, g_{j}}))$ is injective. Now we have
\begin{equation} 
\label{32}
H_{\text{End}^{\text{s}}}^{3}(\mathcal{A}_{2}, \mathbb{V}(\lambda)) = \bigoplus_{\pi \in \mathcal{C}  } m(\pi)  (\pi_{fin})^{K_{0}}. 
\end{equation}
This shows that we construct a subspace of dimension $s_{l + m + 4} s_{l - m + 2}$.\\

\noindent But we can start with a pair $(f_{i}, \overline{g_{j}})$. It means that the $\overline{g_{j}} \in \overline{S_{l - m + 2}}(\Gamma_{1}(1))$. Fix the pair $(f_{i}, \overline{g_{j}})$. We assign to $(f_{i}, \overline{g_{j}})$ the cuspidal automorphic representation  $\pi = \pi(\tau_{f_{i}}, \tau_{\overline{g_{j}}})$ of $GSO(2, 2)(\mathbb{A})$. $\sigma_{f_{i}, \overline{g_{j}}}^{+}$ denotes the extension of $\pi$ to $GO(2, 2)$. Because $(\sigma_{f_{i}, \overline{g_{j}}})_{\infty}^{+} \cong \pi_{l + m + 4, m - l - 2}$, $\sigma_{f_{i}, \overline{g_{j}}} \in \mathcal{A}_{0}(H)$. A similar argument shows that $\Theta(\sigma_{f_{i}, \overline{g_{j}}})$, defined in  (\ref{15}), assign a unique element in the strict endoscopic part of $H^{3}_{!}(\mathcal{A}_{2}, \mathbb{V}(\lambda))$. Explicitly, we construct another copy of $\Theta((\sigma^{+}_{f_{i}, g_{j}}))$, i.e., it appears with multiplicity two in the strict endoscopic part. Note that the dimension of $\overline{S_{l - m + 2}}(\Gamma_{1}(1))$ is equal to dimension of $S_{l - m + 2}(\Gamma_{1}(1))$. This provides us another subspace of dimension $s_{l + m + 4} s_{l - m + 2}$ in the strict endoscopic part. Therefore, we construct a subspace of dimension of $2s_{l + m + 4} s_{l - m + 2}$ in the strict endoscopic part as we desired.
\renewcommand{\qedsymbol}{} 
\end{proof}

The Theorem \ref{MT} shows that the machinery that is presented in this paper can provide us an explicit description of the strict endoscopic part as long as there is enough information on the Euler characteristic of inner cohomology.

\section{Interpretation of the constructed classes in terms of actions of Hecke operators}
\label{sec10}

This section contains some general comments to interpret the results in this paper in terms of Hecke eigenspaces. Explicitly, we want to show that if we consider the generalization of Hecke operators for paramodular representations as in \cite{16}, then our construction can be described in terms of mapping Hecke eigenvectors to Hecke eigenvectors.\\

Let $(\tau, V_{\tau})$ be a global cuspidal automorphic representation of $GL(2, \mathbb{A})$ such that $\tau = \otimes_{\upsilon} \tau_{\upsilon}$ and $V_{\tau} = \otimes_{\upsilon} V_{\upsilon}$ over all places $\upsilon$. We assume that $\tau$ has level $ N_{\tau} = \prod_{\upsilon \in S} p^{a(\tau_{\upsilon})}$. If $\upsilon$ is a finite place, then for $m \geq 1$ let $K_{2}(m)$ be the subgroup of $GL(2, \mathbb{Z})$ consisting of matrices of the form
\begin{equation}
\begin{pmatrix}
a & b \\
c & d  \end{pmatrix}, \quad \text{ $c \equiv 0$ (mod $p^{m}$)}, \quad d \in 1 + p^{n}. 
\end{equation}
 Classically, for each finite place $\upsilon$ we define local Hecke operators at $\upsilon$ by 

\begin{equation*}
T_{1, \upsilon}:= \text{characteristic function of} \hspace{1cm} K_{2}(m)\begin{pmatrix}
\varpi & 0\\
0 & 1
\end{pmatrix} K_{2}(m)
\end{equation*}

\begin{equation*}
T_{1, \upsilon}^{*}:= \text{characteristic function of} \hspace{1cm} K_{2}(m)\begin{pmatrix}
1 & 0\\
0 & \varpi
\end{pmatrix} K_{2}(m),
\end{equation*}  
where act by convolution on $V_{\upsilon}$. According to Theorem \ref{th6}, we know that $\dim(V^{(a(\tau_{\upsilon}))}) = 1$. It is a well-known classical result that the Hecke operators $T_{1, \upsilon}$ and $T_{1, \upsilon}^{*}$ act as endomorphisms on      
$V^{(a(\tau_{\upsilon}))}$. Because it is a one dimensional space, therefore, its elements are Hecke eigenvectors for Hecke operators at place $\upsilon$.\\

B. Roberts and R. Schmidt generalized the concept of Hecke operators for paramodular representations in Chapter 6 in \cite{16}. Briefly, we review the main idea here. Let $(\pi, V_{\pi})$ be a global cuspidal automorphic representation of $GSp(4, \mathbb{A})$ such that $\pi = \otimes_{\upsilon} \pi_{\upsilon}$ and $V_{\pi} = \otimes_{\upsilon} V_{\upsilon}$ over all places $\upsilon$. We assume that $\pi$ is a paramodular representation of level $ N_{\pi} = \prod_{\upsilon \in S} p^{a(\pi_{\upsilon})}$ as defined in Section \ref{sec7}. If $\upsilon$ is a finite place, then for $m \geq 1$ let $K(\mathcal{P}^{m})$ be the paramodular level subgroup at place $\upsilon$. For each of finite places $\upsilon$, Hecke operators at $\upsilon$ were defined by B. Roberts and R. Schmidt in \cite{16} at page 189 as follows.

\begin{equation*}
T_{0, 1, \upsilon}:= \text{characteristic function of} \hspace{1cm} K(\mathcal{P}^{m})\begin{pmatrix}
\varpi & 0 & 0 & 0\\
0 & \varpi & 0 & 0\\
0 & 0 & 1 &0 \\
0 & 0 & 0 & 1
\end{pmatrix} K(\mathcal{P}^{m})
\end{equation*}

\begin{equation*}
T_{1, 0, \upsilon}:= \text{characteristic function of} \hspace{1cm} K(\mathcal{P}^{m})\begin{pmatrix}
\varpi^{2} & 0 & 0 & 0\\
0 & \varpi & 0 & 0\\
0 & 0 & \varpi &0 \\
0 & 0 & 0 & 1
\end{pmatrix} K(\mathcal{P}^{m}),
\end{equation*}  
where act by convolution on $V_{\upsilon}$. According to Theorem 7.4.4. in \cite{16} at page 259, we know that $\dim(V^{(a(\pi_{\upsilon}))}) = 1$ and the Hecke operators $T_{0, 1, \upsilon}$ and $T_{1, 0, \upsilon}$ act as endomorphisms on      
$V^{(a(\pi_{\upsilon}))}$. Because it is a one dimensional space, therefore, its elements are Hecke eigenvectors for Hecke operators at place $\upsilon$. We are following \cite{16} to denote the eigenvalues of $T_{0, 1, \upsilon}$ and $T_{1, 0, \upsilon}$  by $\lambda$ and $\mu$ respectively. In summary, we say that for any paramodular representation $\pi$ of $GSp(4, \mathbb{A})$ with trivial central character, if it is level is $N_{\pi}$, then the elements that are fixed by paramodular level subgroup at minimal level for each place $\upsilon$ are Hecke eigenvectors for Hecke operators $T_{0, 1, \upsilon}$ and $T_{1, 0, \upsilon}$ at corresponded place $\upsilon$.\\

Now let $f$ and $g$ be newforms in $S_{l+m+4}(\Gamma_{1}(N_{1}))$ and $S_{l-m+2}(\Gamma_{1}(N_{2}))$ respectively (i.e., of levels $N_{1}$ and $N_{2}$ respectively) with trivial central character. If $\pi = \pi(\tau_{f}, \tau_{g})$ is the assigned cuspidal automorphic representation of $GSO(2, 2)(\mathbb{A})$ to the pair $(f, g)$, then precisely, $(f, g)$ is assigned to a pair of vectors $(v_{1}, v_{2})$ such that $v_{1} \in V_{\tau_{f}}$, $v_{2} \in V_{\tau_{g}}$, $v_{1} = \otimes_{\upsilon}v_{1, \upsilon}$, $v_{2} = \otimes_{\upsilon}v_{2, \upsilon}$, and $v_{i, \upsilon} \in V^{a(\tau_{i, \upsilon})}$ for $i = 1$ or $2$. The classical newforms $f$ and $g$ are Hecke eigenfunctions for the classical Hecke operators over modular forms (for example see Chapter one in \cite{3}). This means that they are corresponded to $v_{1}$ and $v_{2}$ which are restricted direct product of the Hecke eigenvectors for adelic descriptions of classical Hecke operators i.e., $T_{1, \upsilon}$ and $T_{1, \upsilon}^{*}$.\\

Now we can follow our construction of strict endoscopic part by using global theta lift of $\pi = \pi(\tau_{f}, \tau_{g})$ to $GSp(4, \mathbb{A})$. More precisely, if the $\Theta(\pi(\tau_{f}, \tau_{g}))$ has paramodular level $N$ according to our analytic investigation in Section \ref{sec7}, then
\begin{equation}
\label{1001}
 (f, g) \longleftrightarrow \pi(\tau_{f}, \tau_{g}) \longmapsto (\Pi_{\sigma^{+}, fin})^{K(N)}
\end{equation}

\noindent assigns to $(f, g)$ the strict endoscopic cohmological class $(\Pi_{\sigma^{+}, fin})^{K_{N}}$ under the isomorphism (\ref{4}). Note that we know that $\dim((\Pi_{\sigma^{+}, fin})^{K_{N}}) = 1$ because the paramodular representation $\Pi_{\sigma^{+}}$ has level $N$. Moreover, we know that the elements of  $(\Pi_{\sigma^{+}, fin})^{K_{N}}$ are restricted direct product of eigenvectors of local Hecke operators $T_{0, 1, \upsilon}$ and $T_{1, 0, \upsilon}$ for each finite place $\upsilon$. This means our construction in (\ref{1001}) is the matching of Hecke eigenfunctions $f$ and $g$ (of level $N_{1}$ and $N_{2}$ respectively) to one dimensional Hecke eigenspace  $(\Pi_{\sigma^{+}, fin})^{K_{N}}$, where $N$ determined according to analytic cases in Section \ref{sec7}. We should mention that B. Roberts and R. Schmidt precisely calculate the eigenvalues $\lambda$ and $\mu$ for local papramodular representations of $GSp(4,F _{\upsilon})$ for each finite place $\upsilon$ in Table A. 9 in \cite{16}. Therefore, according to analytic cases at Section \ref{sec8}, we are able to determine the eigenvalues $\lambda$ and $\mu$ in terms of $\lambda_{f, \upsilon}$ and $\lambda_{g, \upsilon}$ which denoted the eigenvalues of $f$ and $g$ respectively (although we skip the calculation here.) In summary the approach to attack Conjecture \ref{conj1} in this paper provides us extra explicit data in terms of actions of Hecke operators. In fact the author believes that this provides a method to relate the $L$-packets of $\pi$ to $\Pi_{\sigma^{+}}$ in more precise way. In other words the results in this paper should have interpretation in terms of $L$-packets too although this is a topic for further research project.

\begin{acknowledgements}
This paper was the Ph. D project of the author under supervision of S. S. Kudla. The author thanks him for all his supports during this project. The author also wish to thank J. Arthur, R. Schmidt, J. Schwermer, M. Shahshahani and  S. Takeda for helpful discussions. 
\end{acknowledgements}

\bibliographystyle{amsalpha}
\bibliography{OnthestrictendoscopicpartofmodularSiegel}

\end{document}